\documentclass[onefignum,onetabnum]{siamart250211}

\headers{Local error estimates for (non-)linear stabilization FEM}{E. Burman, F. Heimann}

\title{Local error estimates for a finite element method combining linear and nonlinear stabilization for the linear hyperbolic transport equation\thanks{Submitted to the editors DATE.
}}

\author{Erik Burman\thanks{Department of Mathematics, University College London, UK, \texttt{e.burman@ucl.ac.uk}} \and
Fabian Heimann\thanks{Department of Mathematics, University College London, UK, \texttt{f.heimann@ucl.ac.uk}}}

\usepackage{lipsum}
\usepackage{amsfonts}
\usepackage{amsmath}
\usepackage{amssymb}
\usepackage{nicefrac}
\usepackage{graphicx}
\usepackage[caption=false]{subfig}
\usepackage{epstopdf}
\usepackage{algorithmic}
\ifpdf
  \DeclareGraphicsExtensions{.eps,.pdf,.png,.jpg}
\else
  \DeclareGraphicsExtensions{.eps}
\fi

\usepackage[normalem]{ulem}

\newsiamremark{remark}{Remark}
\newsiamremark{assumption}{Assumption}
\crefname{assumption}{Assumption}{Assumptions}
\newsiamremark{hypothesis}{Hypothesis}
\crefname{hypothesis}{Hypothesis}{Hypotheses}
\AddToHook{env/remark/begin}{\crefalias{theorem}{remark}}
\AddToHook{env/assumption/begin}{\crefalias{theorem}{assumption}}
\AddToHook{env/hypothesis/begin}{\crefalias{theorem}{hypothesis}}
\newsiamthm{claim}{Claim}
\newsiamthm{conjecture}{Conjecture}
\crefname{claim}{Claim}{Claims}
\crefname{conjecture}{Conjecture}{Conjectures}
\newsiamthm{attempt}{Attempt}

\usepackage{amsopn}

\usepackage[]{hyperref}
\usepackage[capitalize]{cleveref}

\providecommand{\vertiii}[1]{{\left\vert\kern-0.25ex\left\vert\kern-0.25ex\left\vert #1
\right\vert\kern-0.25ex\right\vert\kern-0.25ex\right\vert}}

\usepackage{tikz}
\usepackage{tikz-3dplot}
\usetikzlibrary{fpu,positioning,shapes,arrows,calc, arrows.meta, patterns, intersections, decorations.pathreplacing,calligraphy}
\usetikzlibrary{shapes,arrows,positioning,fit,calc,spy, colorbrewer}

\usepackage{pgfplots}
\usepgfplotslibrary{colorbrewer,groupplots}

\pgfplotsset{
cycle list/Set1-5,
cycle multiindex* list={
mark list*\nextlist
Set1-5\nextlist
},
}

\usepackage{shellesc}
\usetikzlibrary{external}

\usepackage{booktabs}

\usepackage{graphics}

\usepackage{rotating}
\usepackage{mathtools}

\providecommand{\bi}[1]{\hyperlink{def:bi}{b_{i}}}

\providecommand{\ThnG}[1]{\hyperref[eq:new_st_isoparam_regions]{\mathcal{T}_{h,#1}^{\Gamma}}}

\usepackage{pifont}

\ifpdf
\hypersetup{
  pdftitle={Local error estimates for (non-)linear stabilization FEM},
  pdfauthor={E. Burman, F. Heimann}
}
\fi

\newcommand{\jump}[1]{[ #1 ]}
\newcommand{\tripplenorm}[1]{||| #1 |||}

\begin{document}

\maketitle

\begin{abstract}
In this paper, we investigate the combination of a linear continuous interior penalty type and a non-linear artificial diffusion stabilisation applied to the transport problem, based on continuous Galerkin finite elements in space. This method was recently introduced and analysed for globally smooth solutions in [Burman 2023, SIAM J. Sci. Comput., 45, 1, A96-A122]. We provide a rigorous proof of a localisation principle in terms of weighted stability and a priori error bound results, which follow the widely known $\mathcal{O}(h^{k+1/2})$ scaling in the $L^2(\Omega)$ norm at the final time $T$, where $k$ denotes the polynomial order of the finite element space and $h$ the mesh size, provided that the switch between the stabilisations is sufficiently sharp. The analysis is semi-discrete in space and assumes sufficient local regularity of the continuous solution on the smooth part of the domain, where the continuous interior penalty stabilisation is active, whilst artificial diffusion operates on the remaining rough parts of the domain. Thereby, the analysis demonstrates that typical numerical errors in the rough part stay localised relative to the convection velocity and do not negatively affect the smooth parts of the solution, if the stabilisation combination is set up accordingly.
\end{abstract}



\section{Introduction}
Finite element methods (FEM) are a widely-used tool for the simulation of physical phenomena of interest, including flow simulations \cite{brennerscott, hughes2012finite, ern2021finite}. A variety of specific partial differential equations is salient for different applications. In this paper, we focus on the transport or convection equation, which models the transport of a scalar species by a predefined convection velocity, with no diffusion. The transport equation is mathematically interesting both in its own right, but also as a limit case of convection-diffusion systems with dominating convection.

For such applications, it is instructive to distinguish continuous Galerkin (CG) and discontinuous Galerkin (DG) finite element methods \cite{ern2021finite, di2011mathematical}. Both share the same choice of polynomials of order $k$ locally on each element $T$ to construct a discrete space, in which the solution to the discrete problem will eventually be constructed. Whilst these elements are kept independent in the DG method, degrees of freedom on shared boundaries of elements will be joint in CG, which leads to continuous discrete functions.

When methods of both types are applied to convection-diffusion problems with vanishing diffusion (or, equivalently, dominating convection), the choices of the DG fluxes as upwind fluxes or penalties on the solution jump have been established to yield optimal stability and accuracy properties, whilst non-dissipative versions of DG show a stability constant diverging with the inverse of the diffusion constant. Similarly, CG methods--resulting in fewer degrees of freedom, and coming without choices of discrete fluxes--benefit from similar stabilisations of e.g. jumps in the gradient to preserve stability in the case of vanishing diffusion \cite[Chapter 61]{ern2021finite3}. There are several options for such additional stabilisation, which include early work on Streamline-Upwind Petrov-Galerkin (SUPG) stabilisation \cite{BROOKS1982199}, and, more recently, Continuous Interior Penalty (CIP) (alternatively called Gradient Jump Penalty, GJP) stabilisations \cite{10.1007/BFb0120591, BURMAN20041437, M2AN_2007__41_1_55_0, burman2007continuous, doi:10.1137/120867482}.

In the recent paper \cite{burman2023someobservations}, a combination of a linear CIP with a non-linear artificial diffusion type stabilisation was introduced, with the motivation of applying the CIP framework to application cases with e.g. shocks in some areas of the flow field. In such application cases, it proves beneficial to detect elements with (or close to) a shock by a discrete residual estimator, and switch the CIP stabilisation off on these elements in order to activate artificial diffusion stabilisation there instead.

Apart from a rigorous mathematical analysis focusing on globally smooth solution functions, \cite{burman2023someobservations} observed numerically that in the presence of a shock, areas of the discrete solution some physical distance away remain largely unaffected by the unavoidable challenges in resolving the shock accurately. The main contribution of the present paper is a rigorous mathematical statement and proof of this property.

For this purpose, we provide a local error estimate for the variant combining linear and non-linear stabilisation, extending the mathematical analysis in \cite{burman2023someobservations}, where global estimates for the stabilisation combination were established.
 
The technique of using weighted norms will be the main technical tool to achieve these estimates. This was developed and applied to the related case of global CIP stabilisation in \cite{burman2022weighted}. Hence, the present paper can be also regarded as a generalisation of \cite{burman2022weighted} to the case of mixed stabilisation.

Local estimates for hyperbolic problems -- capturing the property of numerical errors caused by rough parts of the solution staying localised -- have been developed for other cases in the literature as well. For linear hyperbolic equations in one spatial dimension, discretised by dissipative Galerkin methods, Layton~\cite{layton1981estimates} proved error estimates in regions bounded away from an initial discontinuity. For the streamline diffusion method applied to nonstationary convection-diffusion problems, Zhou~\cite{zhou1995local} established local error estimates of order $O(h^{3/2})$ for piecewise linear elements on subdomains aligned with the characteristic directions, employing exponentially decaying weight functions similar in spirit to the ones below, at the price of an enlargement of the subdomain of size $O(h^{3/4}|\log h|)$. More recently, Cockburn and Guzm\'an~\cite{doi:10.1137/060668936} showed for the Runge--Kutta discontinuous Galerkin method applied to the transport equation with discontinuous initial data that accuracy is retained outside a pollution region of width $O(h^{1/2})$, up to logarithmic factors, around the transported discontinuity. In contrast, we consider continuous Galerkin elements in $d=2,3$ with a residual-based, non-linear switch between stabilisations, and obtain the localisation at scale $O(h^{1/2}|\log h|)$.

To obtain these rigorous local stability and a priori error results for the stabilisation combination, the paper is structured as follows. First, in the remainder of this introduction, we specify in detail our model problem of consideration. Moreover, we show a numerical example of the mathematical structure to be established for increased readability. In the following \Cref{sect_intro_method}, we introduce the spatially discrete method of interest. Afterwards, in \Cref{sect_math_analysis} the rigorous stability and a priori error results are established.

\subsection{Continuous Problem} To be specific, we focus on the following transport problem: Let the time interval of interest be given as $[0, T]$, as well as a polygonal spatial domain $\Omega \subseteq \mathbb{R}^{d}, d=2,3$. Moreover, assume that the initial concentration of a species is known as $u(x,0) = u_0(x), x \in \Omega$, as well as a divergence-free convection field $\beta\colon \Omega \times [0,T] \to \mathbb{R}^d$, $\nabla \cdot \beta = 0$. We distinguish two cases regarding boundary conditions. First, for the case of inflow boundary conditions, we are interested in the function $u = u(x,t)$ which satisfies
\begin{align}
 \mathcal{L}(u) := \partial_t u + \beta \cdot \nabla u &= f \quad \textnormal{ in } \Omega, \\
 u &= g \quad \textnormal{ on } \partial \Omega_{\text{in}}(t), \nonumber
\end{align}
where $f\colon \Omega \times [0,T]$ describes physical sources or sinks of concentration, and $g\colon \partial \Omega_{\text{in}} \times [0,T]$ inflow Dirichlet data on the inflow part of the boundary defined as
\begin{equation}
 \partial \Omega_{\text{in}} (t) := \{ x \in \partial \Omega  \, | \, n_{\partial \Omega} \cdot \beta(x,t) < 0\}.
\end{equation}

Performing partial integration of this strong formulation yields the following continuous weak formulation: Find $u(t)$ such that for all $v(t)$
\begin{equation}
 (\mathcal{L}(u) ,v)_\Omega + (|\beta n_{\partial \Omega}|u,v )_{ \partial \Omega_{\text{in}}(t) } = (f,v)_\Omega + (|\beta n_{\partial \Omega}|g,v )_{ \partial \Omega_{\text{in}}(t) }. \label{eq_cont_problem_inflow}
\end{equation}
This can be also interpreted as a weak imposition of the Dirichlet inflow boundary data. As usual, we denote the $L^2$ inner product for domains $S \subseteq \Omega$ as $(u, v)_S = \int_S \mathrm{d}x u v$.

Second, in the case of periodic boundary conditions, we assume that the boundaries of the domain $\Omega$ are connected in a periodic way, so that the strong form of the problem merely reads $\mathcal{L}(u) = f$ in $\Omega$ and the corresponding weak form becomes:  Find $u(t)$ such that for all $v(t)$
\begin{equation}
 (\mathcal{L}(u) ,v)_\Omega = (f,v)_\Omega. \label{eq_cont_problem}
\end{equation}
\subsection{Motivating example: linear, non-linear, and combined stabilisation for a shock} Let us give a visualisation of an intended application case of the stabilisation combination, and of the practical relevance of localised error estimates for the reader's convenience. Let $\Omega = [0,1]^2$ be the unit square in 2D, $\beta = (1,0)^T$, and an initial concentration given as $u_0(x,y) = 1$ for $x<\frac{1}{3}$ and $u_0 (x,y) = 0$ for $x\geq \frac{1}{3}$. As boundary condition, let $g=1$ on the left hand side edge. Hence, we will observe a shock at the line $x=\frac{1}{3}$, which traverses to the right. We discretise this problem with CG elements of second order in space on an unstructured mesh with $h = 0.1$ and an explicit three-stage Runge--Kutta method in time with $\Delta t = 0.02 h$. We compare three stabilisations, which are defined in detail in \Cref{sect_intro_method}: plain CIP stabilisation ($s_0$), the non-linear artificial diffusion stabilisation on its own, scaled by the residual-based weight $\varpi$ ($s_1$), and the combination of both ($s_0 \& s_1$), with parameters $\sigma_0 = 0.01$, $\sigma_1 = 0.5$, $\alpha = 2$, $U = 0.5$, and $\rho_1 = 1$, $\rho_2 = 0$. Note that the example uses inflow boundary conditions for illustration, whereas the analysis below concerns the periodic case; the localisation effect is the same. Typical results after discrete interpolation, $t=0$, and at $t = 0.375$ are shown in the upper row of \Cref{figure_motivating_example}.
\begin{figure}
 \begin{center}
  \includegraphics[width=\textwidth]{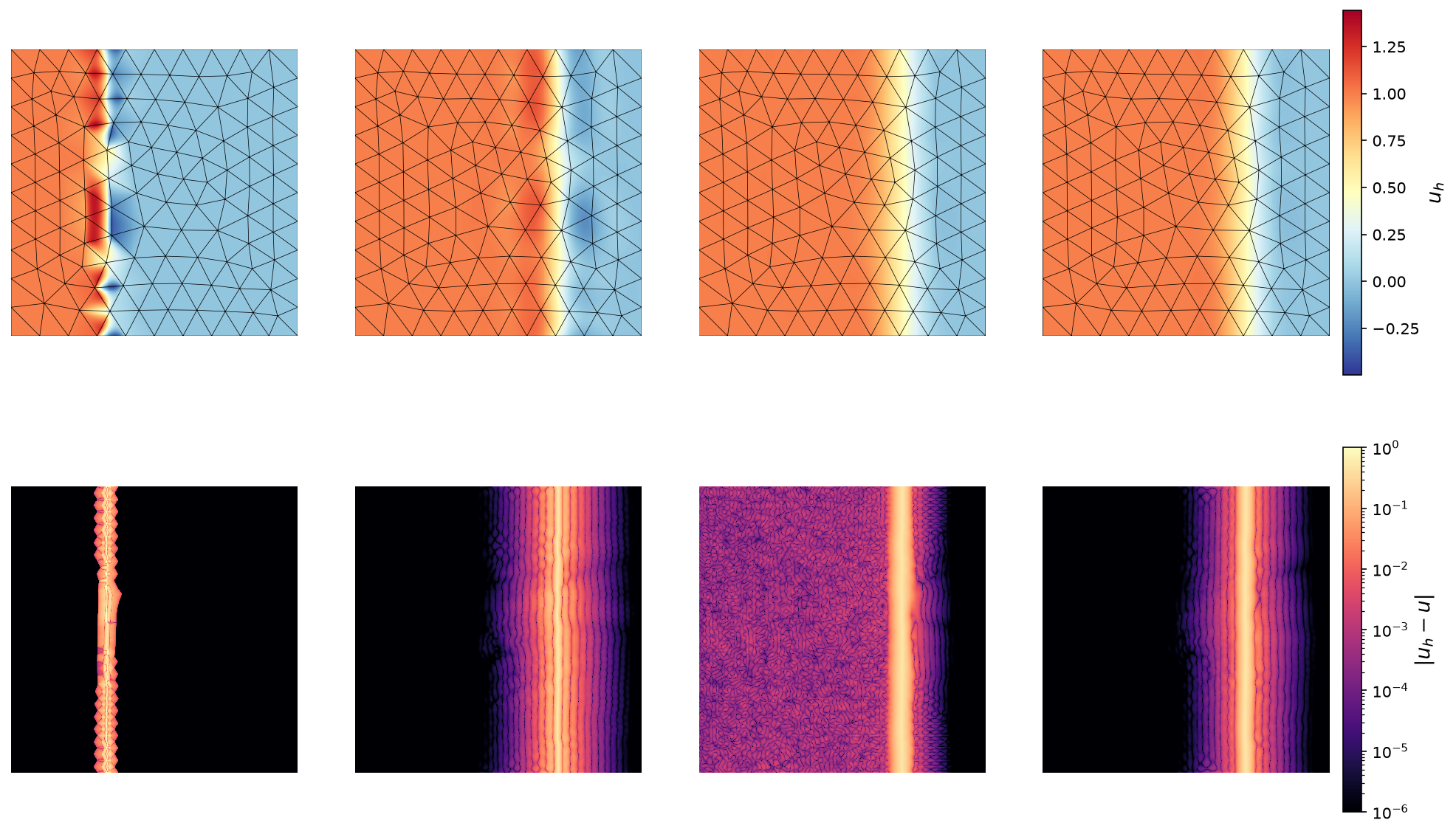}
 \end{center}
 \caption{Transported shock. From left to right: discrete interpolation at $t=0$, and discrete solutions at $t=0.375$ with plain CIP ($s_0$), with the non-linear stabilisation alone ($s_1$), and with the stabilisation combination ($s_0\&s_1$). Upper row: discrete solution $u_h$ for $h = 0.1$, with the mesh. Lower row: pointwise error $|u_h - u|$ in logarithmic scale for $h = 0.025$.}
 \label{figure_motivating_example}
\end{figure}
\begin{table}
 \caption{Transported shock: $L^2$ error at $t = 0.375$ in the smooth part of the domain at distance at least $0.15$ from the shock, $\{ x \in \Omega \, | \, |x_1 - (\frac13 + t)| \geq 0.15 \}$.}
 \label{table_motivating_example}
 \begin{center}
 \begin{tabular}{c|ccc}
  $h$ & $s_0$ & $s_1$ & $s_0 \& s_1$ \\ \hline
  0.1    & 9.13e-3 & 3.17e-3 & 3.76e-3 \\
  0.05   & 1.21e-3 & 8.50e-4 & 3.16e-4 \\
  0.025  & 3.37e-5 & 5.60e-4 & 9.81e-6 \\
  0.0125 & 4.53e-8 & 3.47e-4 & 1.84e-8
 \end{tabular}
 \end{center}
\end{table}
Unsurprisingly, the shock poses a challenge to the discrete interpolation ($t=0$, left hand side), which shows in the over- and undershoots in its vicinity. When progressing in time, all three stabilised discrete solutions transport this area in a stable manner, but with distinct qualities. Plain CIP stabilisation retains pronounced over- and undershoots at the shock, of about $15$--$20\%$ of the jump height. In contrast, the artificial diffusion, which is enabled on elements which are found to show a discrete residual, facilitates an approximation of the shock with over- and undershoots of only a few percent, both for $s_1$ alone and for the combination. This illustrates the motivation for introducing the non-linear stabilisation in application cases with shocks.

Away from the shock, however, the picture is reversed. As visible in the lower row of \Cref{figure_motivating_example}, the error of the non-linear stabilisation alone is polluted in the whole upstream part of the domain: there, the residual-based weight $\varpi$ is (almost) zero, so that the method reduces to an essentially unstabilised Galerkin method, which does not damp the perturbations emitted by the shock. Correspondingly, the error in the smooth part of the domain stagnates under mesh refinement for $s_1$ alone, see \Cref{table_motivating_example}, whereas it decreases rapidly both for plain CIP and for the stabilisation combination. Hence, only the combination of both stabilisations yields an accurate approximation of the shock as well as of the smooth part of the solution. The purpose of the weighted or localised stability and a priori error estimate statements of this paper is the rigorous mathematical description of the latter property.
\section{Spatially Discrete problem with stabilisation combination} \label{sect_intro_method}
\subsection{Defining the discrete problem}
We fix some notation in order to introduce the discrete method: Let $\mathcal{T}_h$ denote a triangulation of the polygonal domain $\Omega$ with maximal mesh size $h$, or maximal diameter $h_T$ for each $T \in \mathcal{T}_h$, taken from a shape-regular and quasi-uniform family of triangulations, i.e.\ $h_T \geq c\, h$ for all $T \in \mathcal{T}_h$ with $c > 0$ independent of $h$, as in \cite{burman2023someobservations}. Quasi-uniformity is used in the global inverse inequalities of the analysis below. Furthermore, we denote the facets of this mesh as $\mathcal{F}_h$, which is assumed to be decomposed into interior and exterior facets, $\mathcal{F}_h = \mathcal{F}_i \cup \mathcal{F}_o$, where $\mathcal{F}_o := \{ F \in \mathcal{F}_h \, | \, F \subseteq \partial \Omega \}$ and $\mathcal{F}_i = \{ F \in \mathcal{F}_h \, | \, \exists T_1, T_2 \in \mathcal{T}_h, T_1 \neq T_2, F = T_1 \cap T_2 \}$. For expressions $A$ and $B$, we write $A \lesssim B$ if there exists a constant $C$ independent of the mesh size $h$ (but possibly depending on the polynomial order $k$ and the shape regularity and quasi-uniformity constants of the mesh) such that $A \leq C \cdot B$.

To introduce further straightforward notation, the norm corresponding to the $L^2$ inner product is written as $\| u \|_S := \left( (u,u)_S \right)^{1/2}$. The related $L^\infty$ norm is denoted by $\| \cdot \|_{\infty, S}$. For sets of codimension 1, such as $F \in \mathcal{F}_h$, we assume that the corresponding inner products and norms are defined with regard to the according integral measure.

Let $V_h$ be the usual continuous discrete space of order $k$, i.e.
\begin{equation}
 V_h := \{ v \in C^0(\Omega) \, | \, v|_T \in \mathcal{P}^k(T) \ \forall T \in \mathcal{T}_h \}.
\end{equation}

We continue with specifying the stabilisation forms which will appear in the discrete problem; these are non-linear in their first argument and bilinear in the last two. We will denote by $s_0(w_h; u_h, v_h)$ the linear CIP stabilisation and by $s_1(w_h; u_h, v_h)$ the non-linear stabilisation of artificial diffusion type.\footnote{We mention in passing that the non-linear stabilisation will be non-linear insofar as the weight in the spatially semi-discrete (but temporally continuous) formulation will be chosen as $u_h$, so that $u_h \mapsto s_1(u_h; u_h, v_h)$ would be non-linear. Arguably, the same property holds with regards to $s_0$. In that way, the names linear and non-linear stabilisation should be rather understood as colloquial names and not in a mathematically substantial manner, following the convention of \cite{burman2023someobservations}.}

At a fixed point in time and for each element $T \in \mathcal{T}_h$, we want to introduce a scalar switch parameter $\varpi(w_h)|_T \in [0,1]$, which can be calculated in dependence of an argument discrete function $w_h$. Value 1 should indicate full non-linear stabilisation, whilst value 0 represents full CIP stabilisation, depending on the amount of numerical residual detected, which would correspond to e.g. a shock. As the function should be constant on each element, it will be discontinuous along element boundaries. Introducing options for calculating facet normal gradient jumps or bulk residuals, $\rho_i \in \{ 0,1\}, i=1,2$, we define, denoting by $\jump{\cdot}|_F$ the jump of a function over an interior facet $F \in \mathcal{F}_i$,
\begin{align}
R_T(w_h) :&= \rho_1 \| \jump{\nabla w_h \cdot n_F} \|_{\infty, \partial T \backslash \partial \Omega} + \rho_2 \| \partial_t w_h + \beta \cdot \nabla w_h - f_h \|_{\infty, T}, \label{eq_def_R_T} \\
 \varpi_T :&= \varpi(w_h)|_T := \Big( \min \Big\{ 1, h_T \frac{R_T(w_h)}{U} \Big\} \Big)^\alpha,  \quad T \in \mathcal{T}_h,
\end{align}
where the shorthand $\varpi_T$ is used whenever the argument $w_h$ is clear from the context, and $\alpha, U > 0$ are scalar parameters with the following motivation: Absolute values in the residual (scaled by $h$) are normalised against $U > 0 $. A small enough $U$ will impose the viscous regularisation everywhere where $R_T \neq 0$. In the case of small $U$ or large $R_T$, the $\min(1,\cdot)$ ensures boundedness by 1 for $\varpi(w_h)$. So, we could regard $U$ as the target value of $R_T(w_h) h$, for which the non-linear stabilisation should be fully activated. The parameter $\alpha \geq 0$ controls the power scaling of the decay of the boundary layer from $\varpi =1$ around the shock towards $\varpi = 0$ outside; the higher $\alpha$, the more narrowly localised will the artificial diffusion be. Typical parameter choices in practice include $\alpha = 1,\dots,4$. By $f_h \in V_h$, we denote some discrete interpolation of $f$ with the interpolation property $\| f - f_h \|_\Omega \leq C h^k \| f\|_{H^k(\Omega)}$.

We illustrate the definitions of $R_T(w_h)$ and $\varpi(w_h)$, in dependence of two example values for $\alpha$, in \Cref{fig_res_and_varpi} by the physical example from \Cref{figure_motivating_example}, for the case $\rho_1 = 1, \rho_2 = 0$.
\begin{figure}
 \begin{center}
 \begin{tikzpicture}
  \node at (-4,4.4) {$R_T(w_h)$}; \node at (0,4.4) {$\varpi(w_h), \alpha = 1$}; \node at (4,4.4) {$\varpi(w_h), \alpha = 4$};
  \node at (0,0) {\includegraphics[width=0.96\textwidth]{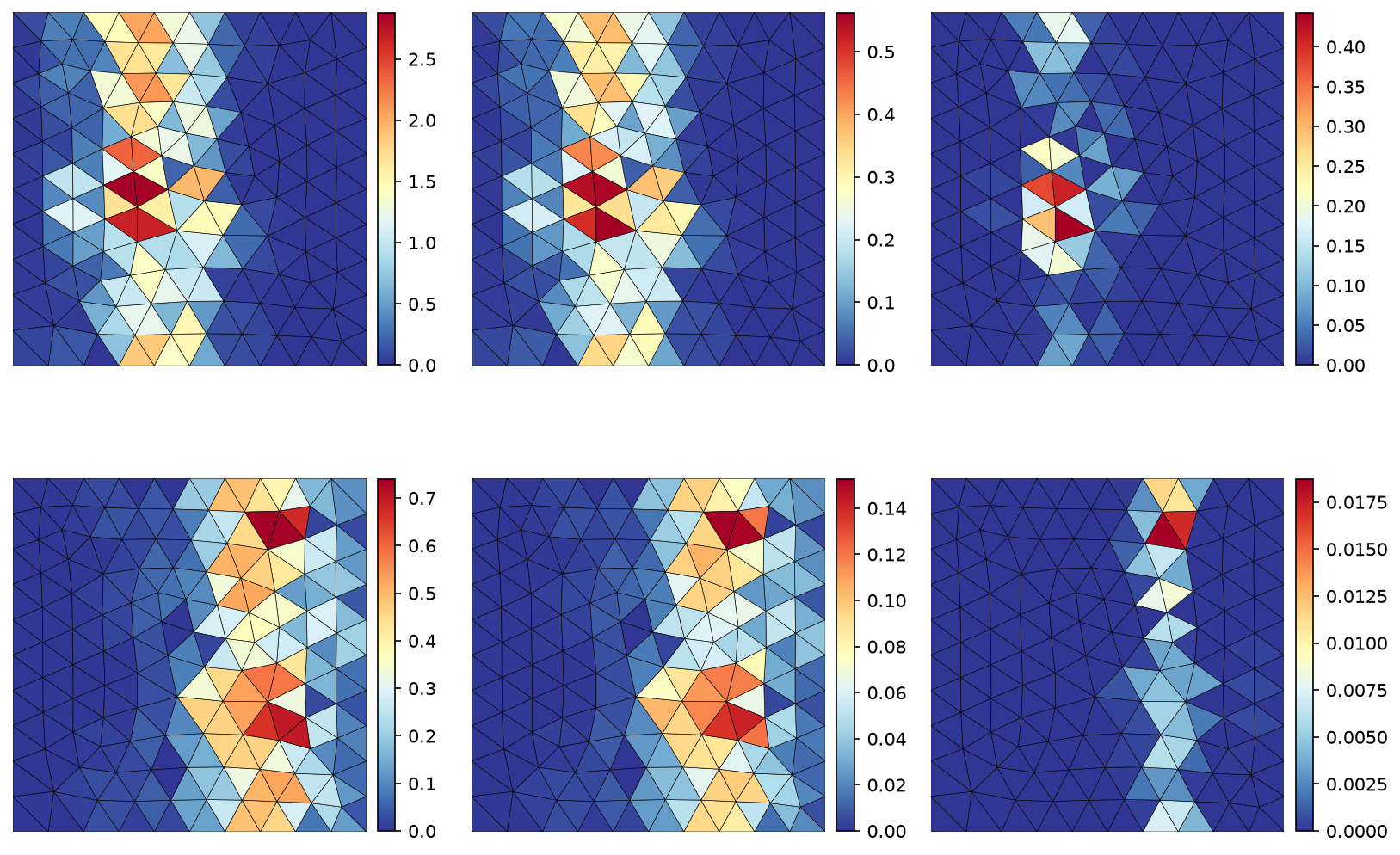}};
  \node[rotate=90] at (-6.4,2.2) {$t=0.05$}; \node[rotate=90] at (-6.4,-2.2) {$t=0.375$};
 \end{tikzpicture}
 \end{center}
 \caption{Illustration of $R_T(w_h)$ and $\varpi(w_h)$ for the stabilisation combination in the example from \Cref{figure_motivating_example}, at $t = 0.05$ in the upper row, and $t=0.375$ in the lower row. In the left column, the element-wise residual $R_T(w_h)$, i.e.\ the maximal normal gradient jump over the element boundary, is shown (for the discrete solution computed with $\alpha = 1$). In the middle column, we show the resulting element-wise constant $\varpi(w_h)$ with $\alpha =1$, which results in a relatively wide region with artificial diffusion stabilisation activated. On the right hand side column, the corresponding $\varpi(w_h)$ calculated with $\alpha = 4$ is shown, which is more narrowly located around the areas of high residual. As the shock is smoothed out over time, the residual decreases, and at $t = 0.375$ the artificial diffusion for $\alpha = 4$ is restricted to a narrow band of elements along the shock front. Note that each individual plot has a distinct color scale.}
 \label{fig_res_and_varpi}
\end{figure}

With the weighting function $\varpi$ defined, we can now proceed with specifying the CIP stabilisation term in areas of small $\varpi$ ($\varpi < 1$) as
\begin{equation}
 s_0(w_h; u_h, v_h) = \sum_{T \in \mathcal{T}_h} h_T^2 \cdot (1 - \varpi_T) \cdot  \left( |\beta| \cdot \jump{\nabla u_h \cdot n_F}, \jump{\nabla v_h \cdot n_F} \right)_{\partial T \backslash \partial \Omega}.
\end{equation}
In this equation, $\jump{\cdot}$ again denotes the jump across the facet $F$ of the boundary $\partial T$, and $n_F$ the outer facet normal to $\partial T$. Note that we have defined $s_0$ in terms of element boundaries and not interior facets in line with the element-wise constant $\varpi(w_h)$. Considering $\partial T \backslash \partial \Omega$ ensures the well-posedness of the jump operator in the case of inflow boundary conditions. In the periodic case, $\Omega$ is understood as a torus: the mesh is assumed to be compatible with the periodic identification, facets on paired boundary portions are identified and treated as interior facets, i.e. contained in $\mathcal{F}_i$, and $\partial T \backslash \partial \Omega$ is to be read as $\partial T$. Likewise, element patches $\omega_T$ and the interpolation operators below are understood across the identified boundary facets. For the analysis, we also introduce the unswitched CIP form
\begin{equation}
 s_{\mathrm{CIP}}(u_h, v_h) := \sum_{T \in \mathcal{T}_h} h_T^2 \left( |\beta| \cdot \jump{\nabla u_h \cdot n_F}, \jump{\nabla v_h \cdot n_F} \right)_{\partial T \backslash \partial \Omega}, \label{eq_def_sCIP}
\end{equation}
so that $s_0(w_h; v_h, v_h) \leq s_{\mathrm{CIP}}(v_h, v_h)$ for all $w_h, v_h \in V_h$. Note that $s_{\mathrm{CIP}}$ does not in general coincide with $s_0(0; \cdot, \cdot)$, since $R_T(0) = \rho_2 \|f_h\|_{\infty,T}$ need not vanish.

The artificial diffusion stabilisation $s_1$ complements $s_0$ in that it scales with $\varpi(w_h)$, in particular
\begin{equation}
 s_1 (w_h; u_h, v_h) = \sum_{T \in \mathcal{T}_h} h_T \cdot \varpi_T \cdot (|\beta| \cdot \nabla u_h, \nabla v_h)_T.
\end{equation}

Finally, we define the physical bilinear form options, following \Cref{eq_cont_problem} distinguishing periodic or inflow boundary conditions:
\begin{equation}
 a(u_h,v_h) := (\mathcal{L} u_h, v_h)_\Omega, \quad a^{\text{in}}(u_h,v_h) := a(u_h, v_h) + (|\beta n_{\partial \Omega}| u_h,v_h )_{\partial \Omega_{\textnormal{in}}(t)}.
\end{equation}
Taking these definitions together, and introducing stabilisation constants $\sigma_0, \sigma_1 > 0$, we arrive at the following weak form problems, which are semi-discrete in space:
\begin{definition}[Discrete problems]
\begin{enumerate}
 \item The discrete problem with inflow boundary conditions is defined as: Find $u_h(t) \in V_h$ such that $\forall v_h(t) \in V_h$
\begin{multline}
 a^{\text{in}}(u_h, v_h) + \sigma_0 s_0(u_h; u_h, v_h) + \sigma_1 s_1(u_h; u_h, v_h) \\
 = (f,v_h)_\Omega + (|\beta n_{\partial \Omega}| g, v_h)_{\partial \Omega_{\textnormal{in}}(t)}. \label{eq_discrete_problem_inflow}
\end{multline}
\item The discrete problem with periodic boundary conditions is defined as: Find $u_h(t) \in V_h$ such that $\forall v_h(t) \in V_h$\footnote{We implicitly assume that in the discrete space, the outer degrees of freedom are coupled according to the periodic boundary.}
\begin{equation}
 a(u_h, v_h) + \sigma_0 s_0(u_h; u_h, v_h) + \sigma_1 s_1(u_h; u_h, v_h) = (f,v_h)_\Omega. \label{eq_discrete_problem}
\end{equation}
\end{enumerate}
\end{definition}
The upcoming rigorous numerical analysis will concern the second case with periodic boundary conditions, but we presented the handling of inflow boundary conditions as well for its computational relevance.
\subsection{Discrete norms and relevant results from the literature}
\subsubsection{Global norms / results from \cite{burman2023someobservations}}
For the analysis, it is instructive to introduce some discrete norms.\footnote{The notation by and large follows \cite{burman2022weighted} and \cite{burman2023someobservations}.} First, we define norms for spatial functions $v_h \in V_h$ or $v_h = v_h (t)$, for time-dependent functions. In particular, we start with the following semi-norm relating to the CIP jumps:
\begin{equation}
 | v_h |_s^2 := s_{\mathrm{CIP}}(v_h, v_h) \sim \sum_{F \in \mathcal{F}_i} h_F^2 |\beta| \, \| \jump{\nabla v_h \cdot n_F} \|_F^2,
\end{equation}
where $h_F$ denotes the diameter of $F$, and the equivalence holds with constants depending only on the shape regularity of the mesh. For obtaining an overall residual norm, we also add volumetric bulk material derivative contributions:
\begin{equation}
 \| v_h \|_R^2 := | v_h |_s^2 + \| h^{1/2}(\partial_t v_h + \beta \cdot \nabla v_h) \|^2_\Omega.
\end{equation}

Next, a variant is defined, which contains furthermore the stabilisation $s_1$ in relation to some weight function $w_h$, or equivalently, localised diffusion contributions as implied by the stabilisation weight:
\begin{equation}
\| v_h \|_{w_h, S}^2 := \| v_h \|_R^2 + s_1(w_h; v_h, v_h) = \| v_h \|_R^2 + \| h^{1/2} \cdot | \beta|^{1/2} \cdot \varpi(w_h)^{1/2} \nabla v_h \|_\Omega^2.
\end{equation}
We note that these norms were introduced in \cite{burman2023someobservations} and used to derive global stability and a priori error estimates for the stabilisation combination of consideration of this paper, for the problem with inflow boundary conditions.\footnote{We mention in passing that to handle specifically the terms stemming from the inflow boundary, the following further norm is considered in \cite{burman2023someobservations} as well $\tripplenorm{v_h}^2_{w_h} := \|v(\cdot, T) \|_\Omega^2 + \int_0^T ( \| | \beta n|^{1/2} v \|^2_{\partial \Omega} + \| v_h \|_{w_h,S}^2) \mathrm{d}t$. We conjecture a similar weighted norm variant (see below) being relevant if a weighted norm analysis of \cref{eq_discrete_problem_inflow} as opposed to \cref{eq_discrete_problem} was put forward.}

Some of the techniques in the proof of the inf-sup stability in \cite{burman2023someobservations} will provide a blueprint for the arguments employed in this paper. For the reader's convenience, we want to briefly reference these cornerstone ideas here, as we believe their structure is more accessible in the less technical global norm case, before we lift them to the weighted norms later:
\begin{enumerate}
 \item Fundamentally, stability is derived in the inf-sup framework, that means for each $v_h \in V_h$, we present a candidate function $w_h$ so that
 \begin{equation}
  (a + \sigma_0 s_0 + \sigma_1 s_1)(v_h, w_h) \gtrsim \tripplenorm{v_h}^2, \textnormal{ for some norm } \tripplenorm{\cdot}.
 \end{equation}
 Insight: It proves instructive to consider a linear combination of the kind $w_h = c_1 v_h + c_2 h \mathcal{L}(v_h)$, with suitable constants $c_1, c_2 > 0$, for this purpose.
 \item As a consequence of that, inner products between $v_h$ and $\mathcal{L}(v_h)$ need to be controlled in the stabilisation terms $s_0$, $s_1$. To this purpose an application of Cauchy-Schwarz yields the following helpful upper bound \cite[Lemma 6]{burman2023someobservations}:
 \begin{lemma} \label{lemma6_burman23}
 It holds for $n \in \{0,1\}$ and all $w_h,v_h \in V_h$, $z \in V_h + H^{3/2 + \epsilon}(\Omega)$
 \begin{equation}
  s_n (w_h; z, h i_{av} \mathcal{L} v_h) \leq C s_n(w_h; z,z)^{1/2} \| h^{1/2} \mathcal{L}(v_h) \|_\Omega. \label{eq_s_n_of_vh_mat_deriv}
 \end{equation}
\end{lemma}
\item Another relevant insight relates to the stabilisation switch. It roughly says that there is never ``too little'' stabilisation, or more specifically: The stabilisation energy which would be employed in the plain CIP case, $s_{\mathrm{CIP}}(v_h, v_h)$, can be bounded from above by the sum of the switched stabilisations, for any weight or switch function \cite[Lemma 3]{burman2023someobservations}:
\begin{lemma} \label{lemma3_burman23}
For all $w_h,v_h \in V_h$ there holds
\begin{equation}
 | v_h |_s^2 \leq C( s_0(w_h; v_h, v_h) + s_1(w_h; v_h, v_h) ). \label{s_0_est_both_s}
\end{equation}
\end{lemma}
\end{enumerate}
\subsubsection{Local norms / results from \cite{burman2022weighted}}
In this paper, we aim at developing a localised version of the results of \cite{burman2023someobservations}. To this end, we introduce now weighted counterparts of the norms defined above. In general, let $\phi$ be a given weighting function, then the weighted norms with an additional $\phi$ lower index are defined as
\begin{align}
 \| v\|_\phi &:= \| \phi v \|_\Omega, \quad \| v_h \|_{R,\phi}^2 := | \phi v_h|_s^2 + \| h^{1/2}(\partial_t v_h + \beta \cdot \nabla v_h) \|^2_\phi \\
 \| v_h \|_{w_h, S, \phi}^2 &:= \| v_h \|_{R,\phi}^2 + \| h^{1/2} \cdot | \beta|^{1/2} \cdot \varpi(w_h)^{1/2} \nabla v_h \|_\phi^2.
\end{align}
The construction of the weighting function $\phi$ follows \cite{burman2022weighted}, adapted to the periodic setting: Let $\varphi \in C^{k+1}([0,\infty))$ be a smooth positive radial profile, depending on the distance $r$ to a centre point $\mathbf{x}_0$, with $\varphi'(r) \leq 0$, $\varphi(r) = 1$ for all $r \leq r_0$ and $\varphi(r) \sim \exp(-(r-r_0)/\sigma)$ for all $r> r_0$ with $\sigma = K \sqrt{h}, K > 1$, and for some $C > 0$,
\begin{equation}
 | \partial^l_r \varphi(r)| \leq C \sigma^{-l} \varphi(r), \quad l \geq 1.  \label{eq:about_varphi}
\end{equation}
Denoting by $\mathbb{L} \subset \mathbb{R}^d$ the lattice of periods of $\Omega$, and assuming that $r_0$ is smaller than half of the minimal period, we define the periodised initial weight
\begin{equation}
 \Phi_h(\mathbf{x}) := \sum_{\mathbf{z} \in \mathbb{L}} \varphi(|\mathbf{x} - \mathbf{x}_0 - \mathbf{z}|).
\end{equation}
Due to the exponential decay of $\varphi$, the series converges, $\Phi_h$ is smooth and $\Omega$-periodic, and, since all summands are positive, \cref{eq:about_varphi} implies $|D^l \Phi_h| \leq C \sigma^{-l} \Phi_h$ for $l \geq 1$. Moreover, $\Phi_h \geq 1$ wherever $|\mathbf{x} - \mathbf{x}_0 - \mathbf{z}| \leq r_0$ for some $\mathbf{z} \in \mathbb{L}$, and $\Phi_h \leq 2$ for $h$ sufficiently small, as the contributions of all but the nearest lattice point are exponentially small in $\sigma^{-1}$. Define $\phi(\mathbf{x},t) := \Phi_h(\mathbf{x} - \beta t)$ on the periodic domain. Then, it follows
\begin{corollary}
 $\phi$ satisfies
 \begin{align}
  \mathcal{L} \phi &= (\partial_t + \beta \cdot \nabla) \phi = 0 \quad \textnormal{ and} \label{eq_phi_transport_L} \\
  | D^l \phi | &\leq C \sigma^{-l} \phi \quad l \geq 1, \quad \textnormal{ in particular }  |\nabla \phi | \leq C K^{-1} h^{-1/2} | \phi | \label{eq_bnd_nabla_phi}.
 \end{align}
\end{corollary}
For an example function $\varphi$, we refer the reader to \cite[Fig. 1]{burman2022weighted}.

An implication of the smoothness of $\phi$ is the following observation.
\begin{corollary}
 For $v_h \in V_h$, it holds
 \begin{equation}
  |\phi v_h|_s^2 = s_{\mathrm{CIP}}(\phi v_h, \phi v_h) = s_{\mathrm{CIP}}(v_h, \phi^2 v_h). \label{eq_switch_arguments_within_s0}
 \end{equation}
\end{corollary}

In relation to the convection velocity, we follow \cite{burman2022weighted} and \cite{burman2023someobservations} in assuming that $\beta$ is constant. This is not intended as a sharp assumption on the results presented, but helps to simplify the presentation.
\begin{assumption} \label{assumption_beta_const}
 We assume that $\beta \in \mathbb{R}^d \setminus \{0\}$ is temporally and spatially constant.
\end{assumption}
We note that a weighted norm analysis of the plain CIP stabilisation, which corresponds to the case of $\rho_1 = \rho_2 = 0$ (and potentially $\sigma_1=0$, although the term is never activated anyhow) in our setting, has been put forward in \cite{burman2022weighted}. The following result of inf-sup stability from this paper serves as a blueprint for our localised stabilisation combination analysis:\footnote{We mention in passing that it refers to the case of periodic boundary conditions. The interesting question of a modification of the result towards inflow boundary conditions is left for future research.}
\begin{lemma}[Weighted stability for plain CIP stabilisation] \label{lemma_plain_CIP}
Let $\sigma_0 > 0$, $K>1$. Assume that $h^{1/2}/ K$ is sufficiently small. For all $v_h \in C^1(0,T; V_h)$ there holds
\begin{align}
 \| v_h (\cdot, T)\|^2_\phi + \sigma_0 \int_0^T | \phi v_h |^2_s \mathrm{d}t \leq &~ \frac{C}{K^2} \int_0^T \| v_h\|_\phi^2 \mathrm{d}t +  \| v_h( \cdot, 0)\|^2_\phi \nonumber \\ & + 2 \int_0^T \left( (\mathcal{L} v_h, w_h)_\Omega + \sigma_0 s_{\mathrm{CIP}}(v_h, w_h)\right) \mathrm{d}t, \label{eq_local_stab_CIP_only}
\end{align}
where $w_h = P_h (\phi^2 v_h)$, with $P_h$ the $L^2$ projection onto $V_h$, and the constant $C \sim \sigma_0 + \sigma_0^{-1}$.
\end{lemma}
\section{Stability and A priori error estimate} \label{sect_math_analysis}
In this section, we obtain stability and a priori error estimates for the mixed stabilisation method introduced above, based on the localised norms, thereby eventually also only requiring local smoothness of the exact solution. Starting with stability, we first summarise results about discrete projection operators from the literature and obtain estimates about the stabilisation bilinear forms, which will be of relevance for the stability argument.
\subsection{Results on projection operators} \label{sect_projection}
In the stability argument, we want to include derivative contributions from a discrete function $v_h \in V_h$, developing on the insight of the usefulness of such contributions for controlling summands such as $\| \mathcal{L}(v_h)\|^2_\Omega$, or $\| \mathcal{L}(v_h)\|^2_\phi$. This will lead to the well-known technical task of mapping a function $\nabla V_h$ to the discrete space. To this end, we use the Oswald projection operator, denoted by $i_{av}\colon V_{DG} \to V_h$, where $V_{DG} := \{ v \in L^2(\Omega) \, | \, v|_T \in \mathcal{P}^k(T) \, \forall T \in \mathcal{T}_h\}$. It is well-known that this operator is stable and satisfies an error bound \cite{burman2007continuous}:
\begin{lemma}
 For the Oswald interpolation operator $i_{av}$ and all $v \in V_{DG}$, it holds
 \begin{equation}
  \| i_{av} (v) \|_\Omega \lesssim \| v \|_\Omega, \quad \| v - i_{av} (v) \|_\Omega^2 \lesssim \sum_{F \in \mathcal{F}_i} h_F \| \jump{v} \|_F^2.
 \end{equation}
\end{lemma}
These results can be strengthened to weighted norm versions.
\begin{lemma}
 For the Oswald interpolation operator $i_{av}$ and all $v \in V_{DG}$, it holds, assuming $h^{1/2}/K$ being sufficiently small,
 \begin{equation}
  \| i_{av} (v) \|_\phi \lesssim \| v \|_\phi, \quad \| v - i_{av} (v) \|_\phi^2 \lesssim \sum_{F \in \mathcal{F}_i} h_F \| \phi \jump{v} \|_F^2. \label{eq_Oswald_properties_local}
 \end{equation}
\end{lemma}
\begin{proof}
    The error difference estimate is proven in \cite{10.1093/imanum/drn001}. For stability, let $\Delta T$ be the set of elements sharing a face/edge or vertex with $T \in \mathcal{T}_h$, then
\[
\| i_{av} (v) \|_{T,\phi} \leq C \| v \|_{\Delta T} \bar \phi^{\Delta T} \leq C(1 + h^{\frac12}/K) \|v \|_{\Delta T,\phi}, \textnormal{ where } \bar \phi^{\Delta T} = \max_{x \in \Delta T} \phi.
\]
\end{proof}

In addition, we want to derive discrete counterparts of weighted versions of discrete functions, very often of $\phi^2 v_h$ for some $v_h \in V_h$, which again poses the challenge of obtaining an appropriate discrete function. To this end, we consider the Clément-type quasi-interpolation projection operator $\pi_h$ of Ern and Guermond, \cite[Chapter 22.3]{ern2021finite}. This differs from the weighted analysis of the plain CIP stabilised method in \cite{burman2022weighted}, where $\pi_h$ denotes the $L^2$ projection; we use the quasi-interpolator since it is locally $W^{1,\infty}$-stable, which is needed for closing the argument in the proof of the final theorem of this paper. Since the quasi-interpolator is not $L^2$-orthogonal, we do not use \Cref{lemma_plain_CIP} directly, but prove the corresponding stability estimate for the stabilisation combination from the weighted energy identity, see \Cref{prop:weight_stab}. The weighted super-approximation property of \cite{burman2022weighted}, on the other hand, carries over to the quasi-interpolator by the same techniques: its proof only uses the weighted stability of the projector, \cref{eq_stab_l2_proj} below, the projection property, which yields $\phi^2 v_h - \pi_h(\phi^2 v_h) = (\mathrm{Id} - \pi_h)(\phi^2 v_h - i_h(\phi^2 v_h))$ for the Lagrange interpolant $i_h$, a super-approximation estimate for $i_h$, and the variation of $\phi$ being small on the scale of a patch, cf. \Cref{eq_bnd_nabla_phi}; see also the discrete commutator property of local projectors in \cite{BERTOLUZZA19991097}.

We assume the Clément-type quasi-interpolation projector to satisfy the following weighted super-approximation property.
\begin{assumption}[Super approximation for $\pi_h \phi^2 v_h$]
 Let $v_h \in V_h$. Assume that $h^{1/2}/ K$ is sufficiently small. Then there holds
 \begin{equation}
  \| \phi^2 v_h - \pi_h(\phi^2 v_h) \|_{\phi^{-1}} + h \| \nabla (\phi^2 v_h - \pi_h(\phi^2 v_h)) \|_{\phi^{-1}} \leq C h^{1/2} K^{-1} \| v_h\|_\phi \label{eq_bnd_phih_res_volume},
 \end{equation}
 and
 \begin{equation}
  \left( \sum_{T \in \mathcal{T}_h} \| \phi^{-1} \nabla (\phi^2 v_h - \pi_h(\phi^2 v_h)) \|_{\partial T}^2 \right)^{1/2} \leq C h^{-1} K^{-1} \| v_h\|_\phi \label{eq_bnd_phih_res_bnd}.
 \end{equation}
\end{assumption}
Also, turning to the special local stability and approximation properties distinguishing the operator from Ern and Guermond from generic $L^2$, we make use of the following results of \cite[Chapter 22.3, Thm 22.6, Cor. 22.8]{ern2021finite}, denoting by $T \in \mathcal{T}_h$ one element, and by $\omega_T$ the patch of neighbouring elements,
\begin{lemma}[Element-local properties of $\pi_h$]
Fix an element $T \in \mathcal{T}_h$, and denote its patch by $\omega_T$. Then, the following interpolation property holds for $v \in H^{r+1}(\omega_T)$
\begin{equation}
  | v - \pi_h v |_{H^m(T)} \lesssim h^{r+1-m} | v |_{H^{r+1}(\omega_T)},
  \qquad 0 \le m \le r+1, \; 0 \le r \le k. \label{EGint_interpol}
\end{equation}
In particular, the operator $\pi_h$ is locally stable, i.e. for all $v \in W^{1,\infty}$ or $L^p(\omega_T)$,
\begin{align}
 | \pi_h v|_{W^{1,\infty}(T)} &\lesssim |v|_{W^{1,\infty}(\omega_T)}, \label{EGint_w1stable} \\
 \| \pi_h v\|_{L^{p}(T)} &\lesssim \|v\|_{L^{p}(\omega_T)}, p =1,2. \label{EGint_Lpstable}
\end{align}
\end{lemma}
\begin{remark}[Frozen weights on patches] \label{rem_weight_freezing}
Denote $\bar \phi_T := \max_{\omega_T} \phi$. For $h^{1/2}/K$ sufficiently
small, \cref{eq_bnd_nabla_phi} implies for every $T \in \mathcal{T}_h$
\begin{equation}
 \bar \phi_T \leq \exp( C \operatorname{diam}(\omega_T) K^{-1} h^{-1/2})
 \min\nolimits_{\omega_T} \phi \leq C' \min\nolimits_{\omega_T} \phi.
 \label{eq_weight_freezing}
\end{equation}
Hence, on any element or facet of $\omega_T$, the weight $\phi$ may be
exchanged with the constant $\bar \phi_T$ at the price of a uniform
constant, and the element-local properties
\cref{EGint_interpol,EGint_w1stable,EGint_Lpstable} imply weighted,
region-summed counterparts by the finite overlap of the patches.
\end{remark}
As a consequence of \cref{eq_weight_freezing} and the element-local properties of $\pi_h$, we obtain the following weighted stability properties.
\begin{corollary}[Weighted stability and projection property of $\pi_h$]
 Assume that $h^{1/2}/K$ is sufficiently small, and let $\psi \in \{\phi, \phi^{-1}\}$. Then it holds for all $v \in H^1(\Omega)$
 \begin{align}
  \| \pi_h v \|_\psi \lesssim \| v\|_\psi, \quad \|\nabla \pi_h v\|_\psi \lesssim \| \nabla v\|_\psi, \quad \| \nabla \pi_h v\|_\psi \lesssim \frac{1}{h} \| v \|_\psi. \label{eq_stab_l2_proj}
 \end{align}
 Moreover, for discrete functions $v_h \in V_h$,
 \begin{align}
  \pi_h v_h = v_h. \label{eq_pi_is_projection}
 \end{align}
\end{corollary}
\begin{proof}
The projection property holds by construction of $\pi_h$, see \cite[Section 22.3]{ern2021finite}. For the stability bounds, fix $T \in \mathcal{T}_h$. By \cref{eq_weight_freezing}, $\max_{\omega_T} \psi \leq C \min_{\omega_T} \psi$ for both choices of $\psi$. Hence, by \cref{EGint_Lpstable} with $p=2$,
\[
 \|\psi\, \pi_h v\|_T \leq \max_{\omega_T} \psi \, \|\pi_h v\|_T \lesssim \max_{\omega_T} \psi \, \|v\|_{\omega_T} \lesssim \|\psi v\|_{\omega_T}.
\]
The gradient bound follows analogously from the local $H^1$-stability $|\pi_h v|_{H^1(T)} \lesssim |v|_{H^1(\omega_T)}$, which is a consequence of \cref{EGint_interpol} with $r=0$, $m=1$, and the triangle inequality. The third bound follows from the first one and a local inverse inequality. Summing over $T \in \mathcal{T}_h$ and using the finite overlap of the patches yields the claim.
\end{proof}
\subsection{Stabilisation estimates}
We begin to set up the stability analysis by the following observation on the price to pay for moving between $\phi^2 v_h$, which is not discrete any more due to the weight, to the discrete quasi-interpolant thereof, in the second argument of stabilisation bilinear terms.
\begin{lemma}
 It holds for all $v_h, y_h \in V_h$,
 \begin{align}
  s_0 (y_h; v_h, \phi^2 v_h - \pi_h \phi^2 v_h) &\leq C \left( s_0(y_h; \phi v_h, \phi v_h)^{1/2} \cdot \frac{|\beta|^{1/2}_\infty}{K} \| v_h \|_\phi \right) \label{eq_s0_second_arg_pih_res_bound} \\
  s_1 (y_h; v_h, \phi^2 v_h - \pi_h \phi^2 v_h) &\leq C \| h^{1/2} |\beta|^{1/2} \varpi(y_h)^{1/2} \nabla v_h \|_\phi  \cdot \frac{|\beta|^{1/2}_\infty}{K} \| v_h \|_\phi. \label{eq_s1_second_arg_pih_res_bound}
 \end{align}
\end{lemma}
\begin{proof}
 \Cref{eq_s0_second_arg_pih_res_bound}: We begin with a Cauchy-Schwarz inequality with recalibrated weighting:
 \begin{align*}
  s_0 (y_h; v_h, \phi^2 v_h - \pi_h \phi^2 v_h) &= s_0 (y_h; \phi v_h, \frac{1}{\phi} ( \phi^2 v_h - \pi_h \phi^2 v_h) ) \\
  &\leq s_0(y_h; \phi v_h, \phi v_h)^{1/2} \cdot | \frac{1}{\phi}( \phi^2 v_h - \pi_h \phi^2 v_h) |_s.
 \end{align*}
 In relation to the second factor, note that by definition of $| \cdot |_s$
 \begin{equation*}
  | \frac{1}{\phi}( \phi^2 v_h - \pi_h \phi^2 v_h) |_s \leq h |\beta|^{1/2}_\infty \left( \sum_{T \in \mathcal{T}_h} \| \jump { \nabla (\frac{1}{\phi}( \phi^2 v_h - \pi_h \phi^2 v_h))} \|_{\partial T}^2 \right)^{1/2}.
 \end{equation*}
 Applying \Cref{eq_bnd_phih_res_bnd}, noting that both $v_h$ and $\nabla \phi^{-1}$ are continuous along element boundaries, then yields the result.
 
 \Cref{eq_s1_second_arg_pih_res_bound}: We begin similarly with Cauchy-Schwarz:
 \begin{multline*}
  s_1 (y_h; v_h, \phi^2 v_h - \pi_h \phi^2 v_h) \\
  \leq \| h^{1/2} |\beta|^{1/2} \varpi(y_h)^{1/2} \nabla v_h \|_\phi \left\| \frac{h^{1/2} |\beta|^{1/2}}{\phi} \nabla (\phi^2 v_h - \pi_h \phi^2 v_h)\right\|_\Omega.
 \end{multline*}
 In relation to the second term, we finish the proof by observing
 \begin{align*}
  \left\| \frac{h^{1/2} |\beta|^{1/2}}{\phi} \nabla (\phi^2 v_h - \pi_h \phi^2 v_h)\right\|_\Omega &\leq h^{1/2} |\beta|_\infty^{1/2} \| \nabla (\phi^2 v_h - \pi_h \phi^2 v_h)\|_{\phi^{-1}} \\
  &\lesssim \frac{|\beta|_\infty^{1/2}}{K} \| v_h \|_\phi \quad \textnormal{ by } \cref{eq_bnd_phih_res_volume}.
 \end{align*}
\end{proof}
Next, we continue with the localised counterparts of \Cref{lemma6_burman23}, i.e. the control based on $\mathcal{L}(v_h)$-ish second arguments in the stabilisation bilinear forms.
\begin{lemma}
 It holds for all $v_h \in V_h$
 \begin{align}
  |s_{\mathrm{CIP}}(v_h, \theta h \pi_h \phi^2 i_{av} \mathcal{L} (v_h))| \lesssim ~ \theta^{1/2} s_{\mathrm{CIP}}(v_h, \phi^2 v_h) + \theta^{3/2} |\beta|_\infty \| h^{1/2} \mathcal{L}(v_h)\|^2_\phi. \label{eq_est_resid_in_second_arg_s0}
 \end{align}
\end{lemma}
\begin{proof}
 We start with a Cauchy-Schwarz inequality
 \begin{equation*}
  |s_{\mathrm{CIP}}(v_h, \theta h \pi_h \phi^2 i_{av} \mathcal{L} (v_h))| \leq |\theta^{1/4} \phi v_h|_s \cdot |\theta^{3/4} \phi^{-1} h \pi_h \phi^2 i_{av} \mathcal{L} (v_h)|_s.
 \end{equation*}
 We apply Young's inequality now and focus on the second expression
 \begin{align*}
  |\theta^{3/4} \phi^{-1} h \pi_h \phi^2 i_{av} \mathcal{L} (v_h)|_s^2 &\leq h^4 \theta^{3/2} |\beta|_\infty \sum_{T \in \mathcal{T}_h} \| \jump{\nabla \phi^{-1} \pi_h \phi^2 i_{av} \mathcal{L} (v_h)} \|_{\partial T \backslash \partial \Omega}^2 \\
  &\leq h^4 \theta^{3/2} |\beta|_\infty \sum_{T \in \mathcal{T}_h} \| \jump{\phi^{-1} \nabla \pi_h \phi^2 i_{av} \mathcal{L} (v_h)} \|_{\partial T \backslash \partial \Omega}^2,
 \end{align*}
 as $\phi$ is continuous. Then, by standard trace and discrete inverse inequality
 \begin{align*}
  & |\theta^{3/4} \phi^{-1} h \pi_h \phi^2 i_{av} \mathcal{L} (v_h)|_s^2 \lesssim h^3 \theta^{3/2} |\beta|_\infty \| \nabla \pi_h \phi^2 i_{av} \mathcal{L}(v_h)\|_{\phi^{-1}}^2 \\
  \lesssim & ~ h \theta^{3/2} |\beta|_\infty \| i_{av} \mathcal{L}(v_h)\|_\phi^2 \lesssim \theta^{3/2} |\beta|_\infty \| h^{1/2} \mathcal{L}(v_h)\|^2_\phi \textnormal{ by } \cref{eq_stab_l2_proj}, \cref{eq_Oswald_properties_local}.
 \end{align*}

\end{proof}

\begin{lemma} It holds for all $v_h, y_h \in V_h$
 \begin{align}
  & |s_1 (y_h; v_h, \theta h \pi_h \phi^2 i_{av} \mathcal{L} (v_h))| \nonumber \\
  \lesssim & ~ \theta^{1/2} \| h^{1/2} | \beta |^{1/2} \varpi(y_h)^{1/2} \nabla v_h \|^2_\phi + \theta^{3/2} |\beta|_\infty (1+ h/K^2) \| h^{1/2} \mathcal{L}(v_h)\|^2_\phi. \label{eq_est_resid_in_second_arg_s1}
 \end{align}
\end{lemma}
\begin{proof}
 We start by writing out the given expression and apply Cauchy-Schwarz
 \begin{align*}
  & |s_1 (y_h; v_h, \theta h \pi_h \phi^2 i_{av} \mathcal{L} (v_h))| \\
  = &| (\theta^{1/4} h^{1/2} | \beta|^{1/2} \varpi(y_h)^{1/2} \phi \nabla v_h, \theta^{3/4} h^{3/2} | \beta|^{1/2} \varpi(y_h)^{1/2} \phi^{-1} \nabla \pi_h (\phi^2 i_{av} \mathcal{L}(v_h)))_\Omega |.
 \end{align*}
 We apply Young's inequality on this term. Noting that $\varpi(y_h) \leq 1$, we observe in relation to the second term
 \begin{align*}
  & \|\theta^{3/4} h^{3/2} | \beta|^{1/2} \varpi(y_h)^{1/2} \nabla \pi_h (\phi^2 i_{av} \mathcal{L}(v_h))\|_{\phi^{-1}}^2  \leq  h^3 \theta^{3/2} | \beta|_\infty \|\nabla \pi_h (\phi^2 i_{av} \mathcal{L}(v_h))\|_{\phi^{-1}}^2 \\
  \lesssim & ~ h^3 \theta^{3/2} | \beta|_\infty \|\nabla (\phi^2 i_{av} \mathcal{L}(v_h))\|_{\phi^{-1}}^2 \\
  \leq & ~ 2 h^3 \theta^{3/2} |\beta|_\infty \| \nabla i_{av} \mathcal{L}(v_h) \|_\phi^2 + 8 h^3 \theta^{3/2} |\beta|_\infty \| (\nabla \phi) i_{av} \mathcal{L}(v_h)\|_{\Omega}^2.
 \end{align*}
 Here, we used the weighted $H^1$-stability \cref{eq_stab_l2_proj} of $\pi_h$ with $\psi = \phi^{-1}$ and the product rule $\nabla (\phi^2 w) = \phi^2 \nabla w + 2 \phi (\nabla \phi) w$. For the first summand, we apply a local inverse inequality, and for the second one \Cref{eq_bnd_nabla_phi}, so that overall
 \begin{align*}
  &h^3 \theta^{3/2} | \beta|_\infty \|\nabla (\phi^2 i_{av} \mathcal{L}(v_h))\|_{\phi^{-1}}^2 \\ \lesssim ~ & h \theta^{3/2} | \beta|_\infty \| i_{av} \mathcal{L}(v_h) \|_\phi^2 + \frac{h^2 |\beta|_\infty}{K^2} \theta^{3/2} \| i_{av} \mathcal{L}(v_h) \|_\phi^2 \\
  \lesssim ~ & \theta^{3/2} |\beta|_\infty \left( h + \frac{h^2}{K^2} \right) \| \mathcal{L}(v_h) \|_{\phi}^2 \textnormal{ by } \cref{eq_Oswald_properties_local}.
 \end{align*}
\end{proof}
Next, we can pose the following weighted version of \Cref{lemma3_burman23}, which relates $s_0$ activated everywhere to the actually considered stabilisation combination.
\begin{lemma} It holds for all $v_h, y_h \in V_h$
\begin{equation}
 s_{\mathrm{CIP}}(\phi v_h, \phi v_h) \leq C \left( s_0(y_h; \phi v_h, \phi v_h) + \| h^{1/2} |\beta|^{1/2} \varpi(y_h)^{1/2} \nabla v_h \|_\phi^2 \right). \label{eq_s_0_zero_weight_vs_s_sum_y_h}
\end{equation}
\end{lemma}
\begin{proof} Let $y_h \in V_h$ be an arbitrary fixed weighting function. We follow the structure of \cite[Proof of Lemma 3]{burman2023someobservations} and start by introducing a decomposition by elements with dominating $s_0$ or $s_1$ stabilisation: $\mathcal{T}_- = \{ T \in \mathcal{T}_h \, | \, \varpi(y_h)|_T \leq \frac{1}{2} \}$,  $\mathcal{T}_+ = \{ T \in \mathcal{T}_h \, | \, \varpi(y_h)|_T > \frac{1}{2} \}$
 \begin{equation}
  |\phi v_h|_s^2 = \sum_{T \in \mathcal{T}_-} h^2 |\beta| \| \jump{\phi \nabla v_h} \|_{\partial T\backslash \partial \Omega}^2 + \sum_{T \in \mathcal{T}_+} h^2 |\beta| \| \phi \jump{\nabla v_h} \|_{\partial T\backslash \partial \Omega}^2 = S_1 + S_2,
 \end{equation}
 because of the continuity of $\phi$ along element boundaries. In relation to $S_1$, note that $1 \leq 2 (1 - \varpi(y_h)|_T)$ for all $T \in \mathcal{T}_-$, so that $S_1 \leq 2 s_0(y_h; \phi v_h, \phi v_h)$. Moreover, we split the faces of elements in $\mathcal{T}_+$ into those interior to $\mathcal{T}_+$ and those adjacent to $\mathcal{T}_-$. For this purpose, let on each $T \in \mathcal{T}_+$ denote $\kappa_{\partial T}$ a function which is constant on each facet in $\partial T$. In particular, let $\kappa_{\partial T}$ be 0 on facets adjacent to $\mathcal{T}_-$, and 1 on facets interior to $\mathcal{T}_+$. Then, we can bound
 \begin{align*}
  S_1 + S_2 &\leq 4 s_0 (y_h; \phi v_h, \phi v_h) + 2 \sum_{T \in \mathcal{T}_+} h^2 \varpi (y_h) |\beta| \| \kappa_{\partial T} \phi \nabla v_h \|_{\partial T \backslash \partial \Omega}^2 \\
  &\lesssim s_0 (y_h; \phi v_h, \phi v_h) + \sum_{T \in \mathcal{T}_+} h^2 \varpi (y_h) |\beta| \| \phi \nabla v_h \|_{\partial T}^2,
 \end{align*}
 since over faces where $\kappa_{\partial T} =1$, the factor $\varpi(y_h)|_T \in [1/2, 1]$ cannot vary by more than a factor of 2. We apply a weighted trace inequality on the discrete $\nabla v_h$ and observe
 \begin{align*}
  \sum_{T\in \mathcal{T}_+} h^2 \varpi (y_h) |\beta| \| \phi \nabla v_h \|_{\partial T}^2 \lesssim \sum_{T\in \mathcal{T}_+} h \varpi(y_h) |\beta| \| \phi \nabla v_h \|_{T}^2 \lesssim \| h^{1/2} |\beta|^{1/2} \varpi(y_h)^{1/2} \nabla v_h \|_\phi^2.
 \end{align*}
 This completes the proof.
\end{proof}
We collect some terms in the following lemma for improved overview in the actual inf-sup stability proof.
\begin{lemma}
 Let $v_h, y_h \in V_h$ be arbitrary and define $w_h := \pi_h \phi^2 (v_h + \theta h i_{av} \mathcal{L}(v_h))$. Assume that $h^{1/2}/K$ is sufficiently small. Then, it holds
 \begin{align}
  &\sigma_0 s_0(y_h; v_h, w_h) + \sigma_1 s_1(y_h; v_h, w_h) \label{eq_est_stab_combination} \\
  \geq & ~ \sigma_1 (c_1 - C \theta^{1/2}) \| h^{1/2} |\beta|^{1/2} \varpi(y_h)^{1/2} \nabla v_h \|^2_\phi \nonumber \\
  & + (c_1 \min(\sigma_0, \sigma_1) - C \theta^{1/2} \sigma_0) s_{\mathrm{CIP}}(v_h, \phi^2 v_h)  \nonumber \\
  & - \frac{C}{K^2} \max(\sigma_0, \sigma_1) |\beta|_\infty \| v_h \|_\phi^2 - C \theta^{3/2} \max(\sigma_0, \sigma_1) |\beta|_\infty \| h^{1/2} \mathcal{L}(v_h)\|_\phi^2, \nonumber
 \end{align}
 where $c_1 > 0$ and $C$ are constants independent of $h$, $K$, $\theta$, $\sigma_0$ and $\sigma_1$.
\end{lemma}
\begin{proof}
 We decompose the left hand side expression as follows
 \begin{align*}
  \sigma_0 s_0(y_h; v_h, w_h) + \sigma_1 s_1(y_h; v_h, w_h) = & \sigma_0 s_0(y_h; v_h, \pi_h \phi^2 v_h) + \sigma_1 s_1(y_h; v_h, \pi_h \phi^2 v_h)\\
  & + \sigma_0 s_0(y_h; v_h, \theta h \pi_h \phi^2 i_{av} \mathcal{L}(v_h)) \\
  & + \sigma_1 s_1(y_h; v_h, \theta h \pi_h \phi^2 i_{av} \mathcal{L}(v_h)) \\
  = & I + II + III + IV.
 \end{align*}
 Starting with term $I$:
 \begin{align*}
  I = \sigma_0 s_0(y_h; v_h, \pi_h \phi^2 v_h) \geq \frac{\sigma_0}{2} s_0(y_h; v_h, \phi^2 v_h) - C \sigma_0 \frac{|\beta|_\infty}{K^2} \| v_h \|_\phi^2 \textnormal{ by } \cref{eq_s0_second_arg_pih_res_bound}.
 \end{align*}
 In relation to $II$, we observe
 \begin{align*}
  II &= \sigma_1 s_1(y_h; v_h, \pi_h \phi^2 v_h) = \sigma_1 s_1(y_h; v_h, \phi^2 v_h) - \sigma_1 s_1(y_h; v_h,\phi^2 v_h - \pi_h \phi^2 v_h)\\
  & \geq \sigma_1 s_1(y_h; v_h, \phi^2 v_h) - \sigma_1 \frac{\epsilon}{2} \| h^{1/2} |\beta|^{1/2} \varpi(y_h)^{1/2} \nabla v_h \|^2_\phi - C \sigma_1 \frac{|\beta|_\infty}{2 K^2 \epsilon} \| v_h \|_\phi^2 \textnormal{ by } \cref{eq_s1_second_arg_pih_res_bound}.
 \end{align*}
 Following up in regards to $s_1(y_h; v_h, \phi^2 v_h)$, we observe that due to the product rule $\nabla (\phi^2 v_h) = \phi^2 \nabla v_h + 2 (\phi \nabla \phi) v_h$,
\begin{align*}
 s_1(y_h; v_h, \phi^2 v_h) &= (h |\beta| \varpi(y_h) \nabla v_h, \nabla (\phi^2 v_h))_\Omega \\
 &= \| h^{1/2} |\beta|^{1/2} \varpi(y_h)^{1/2} \nabla v_h\|^2_\phi + (h |\beta| \varpi(y_h) \nabla v_h, 2 \phi \nabla \phi v_h)_\Omega.
\end{align*}
Recalling that $|\nabla \phi| \leq C K^{-1} h^{-1/2} \phi$ (\Cref{eq_bnd_nabla_phi}), and $\varpi(y_h) \leq 1 \Rightarrow \varpi(y_h) \leq \varpi(y_h)^{1/2}$, by Cauchy-Schwarz and Young's inequality
\begin{align*}
 |(h |\beta| \varpi(y_h) \nabla v_h, 2 \phi \nabla \phi v_h)_\Omega| &= |(h |\beta|^{1/2} \varpi(y_h) \phi \nabla v_h, 2 |\beta|^{1/2} \nabla \phi v_h)_\Omega|\\
 &\leq \frac{1}{2} \| h^{1/2} |\beta|^{1/2} \varpi(y_h)^{ 1/2} \nabla v_h\|^2_\phi + \frac{C { { |\beta|_\infty}}}{K^2}  \| v_h \|_\phi^2.
\end{align*}

This implies, together with the previous equation,
\begin{equation*}
 s_1(y_h; v_h, \phi^2 v_h) \geq \frac{1}{2} \| h^{1/2} |\beta|^{1/2} \varpi(y_h)^{1/2} \nabla v_h\|^2_\phi - \frac{C { |\beta|_\infty}}{K^2} \| v_h \|_\phi^2.
\end{equation*}
 Taking into consideration the start of the argument about $II$, we obtain
 \begin{equation*}
  II \geq \frac{\sigma_1}{4} \| h^{1/2} |\beta|^{1/2} \varpi(y_h)^{1/2} \nabla v_h\|^2_\phi - C \sigma_1 \frac{| \beta|_\infty}{K^2} \| v_h \|_\phi^2.
 \end{equation*}
 Combining the estimates about $I$ and $II$, and applying \Cref{eq_s_0_zero_weight_vs_s_sum_y_h},
 \begin{align*}
  I + II \geq & ~ c_1 \sigma_1 \| h^{1/2} |\beta|^{1/2} \varpi(y_h)^{1/2} \nabla v_h \|^2_\phi + c_1 \min(\sigma_0, \sigma_1) s_{\mathrm{CIP}}(v_h, \phi^2 v_h) \\
  & - \frac{C}{K^2} \max(\sigma_0, \sigma_1) |\beta|_\infty \| v_h \|_\phi^2.
 \end{align*}
 In relation to $III$, we obtain by applying \Cref{eq_est_resid_in_second_arg_s0}
 \begin{align*}
  III = &~\sigma_0 s_0(y_h; v_h, \theta h \pi_h \phi^2 i_{av} \mathcal{L}(v_h)) \\
  \geq &~- C \sigma_0 \theta^{1/2} s_{\mathrm{CIP}}(v_h, \phi^2 v_h) - C \sigma_0 \theta^{3/2} |\beta|_\infty \| h^{1/2} \mathcal{L}(v_h)\|^2_\phi,
 \end{align*}
 and accordingly for $IV$, with \Cref{eq_est_resid_in_second_arg_s1},
 \begin{align*}
  IV = &~\sigma_1 \theta s_1(y_h; v_h, h \pi_h \phi^2 i_{av} \mathcal{L}(v_h)) \\
  \geq &~- C \sigma_1 \theta^{1/2} \|h^{1/2} |\beta|^{1/2} \varpi(y_h)^{1/2} \nabla v_h\|_\phi^2 - C \sigma_1 \theta^{3/2} |\beta|_\infty \| h^{1/2} \mathcal{L}(v_h)\|^2_\phi.
 \end{align*}
Combining these results yields the claim.
\end{proof}
\subsection{Stability}
The following property generalises the localised stability estimate towards the mixed stabilisation framework:
\begin{proposition}[Weighted stability] \label{prop:weight_stab}
Let $K>1$. Assume that $h^{\frac12}/K$ is sufficiently small. Then, there exists $\theta > 0$ sufficiently small such that for all $v_h \in C^1(0,T;V_h)$ and $y_h \in V_h$ there holds
\begin{alignat}{1}
\|v_h(\cdot,T)\|_\phi^2  & + C_\theta \int_0^T  \|v_h\|_{y_h,S,\phi}^2
~\mbox{d}t   \leq C/K^2  \int_0^T \|v_h\|_\phi^2 ~\mbox{d}t +  \|v_h(\cdot,0)\|_\phi^2 \nonumber\\
&+ 2 \int_0^T (\mathcal{L} v_h, w_h)_\Omega + \sigma_0 s_0(y_h; v_h,w_h) + \sigma_1 s_1(y_h;v_h,w_h)~\mbox{d}t,
\end{alignat}
where $w_h = \pi_h \phi^2 (v_h + \theta h\,i_{av}(\partial_t v_h + \beta \cdot \nabla v_h))$ and the constant $C \sim \max(\sigma_0, \sigma_1) |\beta|_\infty + \theta^{-1}$.
\end{proposition}
\begin{proof}
We do not invoke \Cref{lemma_plain_CIP}, since its proof in \cite{burman2022weighted} relies on the orthogonality of the $L^2$ projection. Instead, we start from the weighted energy identity. Since $\mathcal{L} \phi = 0$ by \cref{eq_phi_transport_L}, $\nabla \cdot \beta = 0$, and the boundary conditions are periodic, it holds
\begin{equation*}
 (\mathcal{L} v_h, \phi^2 v_h)_\Omega = \frac12 \int_\Omega \mathcal{L}(\phi^2 v_h^2) \, \mathrm{d}x = \frac12 \frac{\mathrm{d}}{\mathrm{d}t} \|v_h\|_\phi^2,
\end{equation*}
and hence
\begin{equation*}
 2 \int_0^T (\mathcal{L} v_h, \phi^2 v_h)_\Omega \, \mathrm{d}t = \|v_h(\cdot, T)\|_\phi^2 - \|v_h(\cdot, 0)\|_\phi^2.
\end{equation*}
We decompose the test function as $w_h = \phi^2 v_h + r_h + \theta h \pi_h \phi^2 i_{av} \mathcal{L}(v_h)$, with $r_h := \pi_h(\phi^2 v_h) - \phi^2 v_h$, and obtain
\begin{multline*}
 2 \int_0^T \left( (\mathcal{L} v_h, w_h)_\Omega + \sigma_0 s_0(y_h; v_h,w_h) + \sigma_1 s_1(y_h;v_h,w_h) \right) \mathrm{d}t \\
 = \|v_h(\cdot, T)\|_\phi^2 - \|v_h(\cdot, 0)\|_\phi^2 + 2 \int_0^T \left( \Delta_{R1} + \Delta_{R2} + \Delta_{R3} \right) \mathrm{d}t,
\end{multline*}
where
\begin{gather*}
 \Delta_{R1} := (\mathcal{L} v_h, r_h)_\Omega, \quad \Delta_{R2} := (\mathcal{L} v_h, \theta h \pi_h \phi^2 i_{av} \mathcal{L}(v_h))_\Omega, \\
 \Delta_{R3} := \sigma_0 s_0(y_h; v_h,w_h) + \sigma_1 s_1(y_h;v_h,w_h).
\end{gather*}
We denote furthermore $\Delta_{L1} := \|h^{1/2} \mathcal{L}(v_h)\|_\phi^2$ and $\Delta_{L2} := \|h^{1/2} |\beta|^{1/2} \varpi(y_h)^{1/2} \nabla v_h\|_\phi^2$, so that, by \cref{eq_switch_arguments_within_s0}, $\|v_h\|_{y_h,S,\phi}^2 = s_{\mathrm{CIP}}(v_h, \phi^2 v_h) + \Delta_{L1} + \Delta_{L2}$. It remains to bound $\Delta_{R1}$, $\Delta_{R2}$, $\Delta_{R3}$ from below.

For the $L^2$ projection, the time-derivative contribution to $\Delta_{R1}$ would vanish by orthogonality. Here, we bound $\Delta_{R1}$ as a whole using the super-approximation estimate \cref{eq_bnd_phih_res_volume} and Young's inequality,
\begin{align*}
 |\Delta_{R1}| \leq \|h^{\frac12} \mathcal{L} v_h\|_\phi \|h^{-\frac12} r_h\|_{\phi^{-1}} &\leq \frac{C}{K} \|h^{\frac12} \mathcal{L} v_h\|_\phi \|v_h\|_\phi \\
 &\leq \frac{\theta}{12} \Delta_{L1} + \frac{C}{\theta K^2} \|v_h\|_\phi^2.
\end{align*}
 It is instructive to evaluate $\Delta_{R2} = (\mathcal{L} v_h, \theta h \pi_h \phi^2 i_{av} \mathcal{L}(v_h))_\Omega $, noting that by \Cref{eq_Oswald_properties_local} and \Cref{eq_bnd_phih_res_volume},
 \begin{align*}
(\mathcal{L}(v_h), \theta h \pi_h \phi^2 i_{av} \mathcal{L}(v_h))_\Omega &= \|\theta^{1/2} h^{\frac12} \mathcal{L}(v_h)\|_\phi^2 + (\mathcal{L}(v_h), \theta h (\pi_h - \mathrm{Id})(\phi^2 i_{av} \mathcal{L} v_h) )_\Omega \\
& \quad + (\mathcal{L}(v_h), \theta \phi^2 h (i_{av} \beta \cdot \nabla v_h-  \beta \cdot \nabla v_h))_\Omega \\
\ge \theta \left( \frac{1}{2} - C \frac{h^{1/2}}{K}\right) & \|h^{\frac12} \mathcal{L}(v_h)\|_\phi^2  - \theta \|h^{1/2} (i_{av} \beta \cdot \nabla v_h-  \beta \cdot \nabla v_h)\|_\phi^{2}.
\end{align*}
Using now the weighted discrete interpolation bound, \Cref{eq_Oswald_properties_local}
\[
\theta \|h^{1/2} (i_{av} \beta \cdot \nabla v_h-  \beta \cdot \nabla v_h)\|_\phi^2 \leq C \theta \sum_{F \in \mathcal{F}_i} h^2 |\beta|^2 \cdot \|\phi \jump{\nabla v_h \cdot n_F} \|_F^2,
\]
taking into consideration the definition of $s_{\mathrm{CIP}}$ and \cref{eq_switch_arguments_within_s0}, we see that
\[
\Delta_{R2} = (\mathcal{L} v_h, \theta h \pi_h \phi^2 i_{av} \mathcal{L} v_h)_\Omega \ge \frac{\theta}{3} \underbrace{\|h^{\frac12} \mathcal{L} v_h\|_\phi^2}_{\Delta_{L1}} - C \theta |\beta|_\infty s_{\mathrm{CIP}}(v_h, \phi^2 v_h).
\]
For $\Delta_{R3}$, we apply \Cref{eq_est_stab_combination}. Collecting the bounds, we arrive at
\begin{align*}
 \Delta_{R1} + \Delta_{R2} + \Delta_{R3} \geq & ~ \sigma_1 (c_1 - C\theta^{1/2}) \Delta_{L2} \\
 & + \left(c_1 \min(\sigma_0, \sigma_1) - C\theta^{1/2} \sigma_0 - C\theta |\beta|_\infty\right) s_{\mathrm{CIP}}(v_h, \phi^2 v_h) \\
  & + \left(\frac{\theta}{4} - C \theta^{3/2} \max(\sigma_0, \sigma_1) |\beta|_\infty \right)\Delta_{L1} \\
  & - C \left( \max(\sigma_0, \sigma_1) |\beta|_\infty + \theta^{-1} \right) K^{-2} \|v_h\|_\phi^2.
\end{align*}
We now choose $\theta$ sufficiently small, depending on $\sigma_0$, $\sigma_1$ and $|\beta|_\infty$, such that $C\theta^{1/2} \leq \frac{c_1}{2}$, $C\theta^{1/2}\sigma_0 + C\theta |\beta|_\infty \leq \frac{c_1}{2}\min(\sigma_0,\sigma_1)$, and $C \theta^{1/2} \max(\sigma_0, \sigma_1) |\beta|_\infty \leq \frac{1}{8}$. Then the three coefficients are bounded from below by $\frac{c_1}{2} \sigma_1$, $\frac{c_1}{2}\min(\sigma_0, \sigma_1)$ and $\frac{\theta}{8}$, respectively. Inserting this into the identity above and rearranging yields the claim with $C_\theta \sim \min(\sigma_0, \sigma_1, \theta)$ and $C \sim \max(\sigma_0, \sigma_1)|\beta|_\infty + \theta^{-1}$.
\end{proof}
\begin{remark}[Interpretation of $\mathcal{L} u$] \label{rem_Lu}
For non-smooth $u$, such as in the presence of a shock, the individual terms $\partial_t u$ and $\beta \cdot \nabla u$ need not be functions, but are in general only measures, see also \Cref{rem_measure_streamline}. Their sum, however, satisfies $\mathcal{L} u = f$ in the sense of distributions, and is therefore an $L^2$-function. Accordingly, for the exact solution $u$ and $v_h \in V_h$, we understand $a(u, v_h) := (f, v_h)_\Omega$, as in \cref{eq_cont_problem}, and expressions such as $\|h^{\frac12} \mathcal{L} u\|_\phi$ or $\|h^{\frac12} \mathcal{L} (u - \pi_h u)\|_\phi$ are well-defined, with $\mathcal{L} u$ identified with $f$, and $\partial_t \pi_h u = \pi_h \partial_t u = \pi_h (f - \beta \cdot \nabla u)$ defined through the element-wise moments of $\pi_h$, cf.\ the proof of \Cref{thm_weighted_error}. We keep the notation $\mathcal{L} u$ for convenience.
\end{remark}
\begin{lemma}[Galerkin Orthogonality]
 Denoting by $u$ the solution to the continuous problem \Cref{eq_cont_problem} and by $u_h$ the solution to the discrete problem \Cref{eq_discrete_problem}, it holds for all $v_h \in V_h$
 \begin{equation}
  0 = a(u - u_h, v_h) - \sigma_0 s_0 (u_h; u_h, v_h) - \sigma_1 s_1 (u_h; u_h, v_h). \label{eq_galerkin_orth}
 \end{equation}
\end{lemma}
\begin{proof}
 We subtract \Cref{eq_cont_problem} and \Cref{eq_discrete_problem}, noting that the right-hand sides are identical, and that $a(u, v_h) = (f, v_h)_\Omega$ is well-defined by \Cref{rem_Lu}.
\end{proof}
We now introduce the following temporally global variant of $\|v_h\|_{w_h, S, \phi}$, which also contains the final time:
\begin{equation}
 \|v_h\|_{w_h, \phi}^2 := \|v_h(\cdot, T) \|_\phi^2 + \int_0^T  \| v_h \|_{w_h,S, \phi}^2 \mathrm{d}t.
\end{equation}
The following continuity estimate is a weighted analogue of \cite[Lemma 9]{burman2023someobservations}. In contrast to the latter, it is formulated directly for the error equation, and the interpolation error only enters through weighted norms. In particular, no global smoothness of $u$ is required; the splitting into the smooth and rough part of the domain is carried out afterwards in the proof of the a priori error estimate.
\begin{lemma}[Weighted continuity] \label{lemma_weighted_continuity}
Let $u$ be the solution of \cref{eq_cont_problem}, $u_h$ the solution of \cref{eq_discrete_problem}, and $e_h := u_h - \pi_h u$ with $e_h(\cdot, 0) = 0$. Let $\theta \in (0,1]$, $w_h := \pi_h \phi^2 (e_h + \theta h\, i_{av} \mathcal{L} e_h)$, and assume that $h^{1/2}/K$ is sufficiently small. Define for $t \in [0,T]$ the interpolation quantity, with $\mathcal{L} u$ understood as in \Cref{rem_Lu},
\begin{multline*}
 \mathcal{I}(t) := \|h^{-\frac12} (u - \pi_h u)\|_\phi^2 + \|h^{\frac12} \mathcal{L} (u - \pi_h u)\|_\phi^2 \\
 + |\pi_h u|_{s,\phi}^2 + \|h^{\frac12} |\beta|^{\frac12} \varpi(u_h)^{\frac12} \nabla \pi_h u\|_\phi^2,
\end{multline*}
evaluated at time $t$. Then, for every $\epsilon \in (0,1]$,
\begin{multline}
 \int_0^T \left( (\mathcal{L} e_h, w_h)_\Omega + \sigma_0 s_0(u_h; e_h, w_h) + \sigma_1 s_1(u_h; e_h, w_h) \right) \mathrm{d}t \\
 \leq \epsilon \|e_h\|_{u_h,\phi}^2 + \frac{C_\epsilon}{K^2} \int_0^T \|e_h\|_\phi^2 \, \mathrm{d}t + C_\epsilon \left( \|(u - \pi_h u)(\cdot, T)\|_\phi^2 + \int_0^T \mathcal{I}(t) \, \mathrm{d}t \right), \label{eq_weighted_continuity}
\end{multline}
where $C_\epsilon \sim \epsilon^{-1}$ is independent of $h$ and $K$, and depends on $\sigma_0$, $\sigma_1$ and $|\beta|_\infty$.
\end{lemma}
\begin{proof}
By Galerkin orthogonality, \cref{eq_galerkin_orth}, and writing $e_h = (u_h - u) + (u - \pi_h u)$, we obtain
\begin{multline*}
 (\mathcal{L} e_h, w_h)_\Omega + \sigma_0 s_0(u_h; e_h,w_h) + \sigma_1 s_1(u_h;e_h,w_h) \\
 = \underbrace{(\mathcal{L} (u - \pi_h u), w_h)_\Omega}_{I} - \sigma_0 \underbrace{s_0(u_h; \pi_h u,w_h)}_{II} - \sigma_1 \underbrace{s_1(u_h;\pi_h u,w_h)}_{III}.
\end{multline*}
First observe that for $I$ we may write
\begin{multline}\label{eq:pert}
(\mathcal{L} (u - \pi_h u), w_h)_\Omega = (\mathcal{L} (u - \pi_h u), \pi_h (\phi^2 e_h) - \phi^2 e_h)_\Omega+(\mathcal{L}(u - \pi_h u),\phi^2 e_h)_\Omega \\
+ (\mathcal{L} (u - \pi_h u), \pi_h \phi^2 \theta h\,i_{av}(\partial_t e_h + \beta \cdot \nabla e_h)).
\end{multline}
Using integration by parts in space and time, $e_h(\cdot, 0) = 0$, and that $\mathcal{L} \phi = 0$, \Cref{eq_phi_transport_L}, we see that in relation to the second summand on the right hand side,
\begin{align*}
 & \left| \int_0^T (\mathcal{L}(u - \pi_h u),\phi^2 e_h)_\Omega \mathrm{d}t \right| \\
 = & \left| ((u - \pi_h u)(\cdot,T), \phi^2 e_h(\cdot,T))_\Omega - \int_0^T (u - \pi_h u, \phi^2 \mathcal{L}e_h)_\Omega \mathrm{d}t \right|\\
 \leq & ~ \epsilon \|e_h(\cdot,T)\|_\phi^2 + \frac{1}{4\epsilon} \|(u - \pi_h u)(\cdot,T)\|_\phi^2 + \int_0^T \|h^{-\frac12} (u - \pi_h u)\|_\phi \|h^{\frac12} \mathcal{L}e_h \|_\phi \mathrm{d}t.
\end{align*}
For the first term of the right hand side of \Cref{eq:pert} we see that
\[
(\mathcal{L} (u - \pi_h u), \pi_h (\phi^2 e_h) - \phi^2 e_h)_\Omega \leq \|h^{\frac12} \mathcal{L} (u - \pi_h u) \phi\|_\Omega \|h^{-\frac12} (\pi_h (\phi^2 e_h) - \phi^2 e_h)\|_{\phi^{-1}}.
\]
Using now the super approximation estimate, \Cref{eq_bnd_phih_res_volume},
\[
\|h^{-\frac12}(\pi_h (\phi^2 e_h) - \phi^2 e_h)\|_{\phi^{-1}} \leq C K^{-1} \|e_h\|_\phi,
\]
we conclude, writing $\pi_h(\phi^2 e_h) = \phi^2 e_h + (\pi_h(\phi^2 e_h) - \phi^2 e_h)$ and using the triangle inequality for the time integral, that
\begin{align*}
&  \left| \int_0^T (\mathcal{L}(u - \pi_h u),\pi_h (\phi^2 e_h) )_\Omega \, \mathrm{d}t \right| - \epsilon \|e_h(\cdot,T)\|_\phi^2 \\
\leq & \int_0^T \frac{C}{\epsilon}\left( \|h^{-\frac12} (u - \pi_h u)\|_\phi^2
+ \|h^{\frac12} \mathcal{L} (u - \pi_h u) \|_\phi^2+K^{-2} \|e_h\|^2_\phi\right) \mathrm{d}t \\
& + \epsilon \int_0^T \|h^{\frac12} \mathcal{L}e_h \|_\phi^2 \, \mathrm{d}t + \frac{1}{4\epsilon} \|(u - \pi_h u)(\cdot,T)\|_\phi^2.
\end{align*}
By \Cref{eq_stab_l2_proj} for $\pi_h$ and \Cref{eq_Oswald_properties_local} for $i_{av}$, we have for the remaining summand of \Cref{eq:pert}
\[
(\mathcal{L} (u - \pi_h u),\pi_h \phi^2 \theta h\,i_{av}(\mathcal{L} e_h ))_\Omega \leq  \frac{C}{\epsilon} \theta \|h^{\frac12} \mathcal{L} (u - \pi_h u) \|_\phi^2+\epsilon \theta \|h^{\frac12} \mathcal{L}e_h \|_\phi^2.
\]
For term $II$, we split $w_h = \pi_h \phi^2 e_h + \theta h \pi_h \phi^2 i_{av} \mathcal{L} e_h$ and write $\pi_h \phi^2 e_h = \phi^2 e_h - (\phi^2 e_h - \pi_h \phi^2 e_h)$. Using $1 - \varpi(u_h) \leq 1$, the Cauchy--Schwarz inequality, \cref{eq_switch_arguments_within_s0}, and the arguments of the proofs of \cref{eq_s0_second_arg_pih_res_bound} and \cref{eq_est_resid_in_second_arg_s0}, with $\pi_h u$ in the first argument, we obtain
\begin{align*}
|II| &\leq C |\pi_h u|_{s,\phi} \left( |e_h|_{s,\phi} + K^{-1} \|e_h\|_\phi + \theta \|h^{\frac12} \mathcal{L} e_h\|_\phi \right) \\
&\leq \frac{C}{\epsilon} |\pi_h u|_{s,\phi}^2 + \epsilon \left( \|e_h\|_{R,\phi}^2 + K^{-2} \|e_h\|^2_\phi \right).
\end{align*}
Similarly, for term $III$, the Cauchy--Schwarz inequality, the product rule together with \cref{eq_bnd_nabla_phi}, \cref{eq_bnd_phih_res_volume}, and the arguments of the proof of \cref{eq_est_resid_in_second_arg_s1} yield
\begin{align*}
|III| &\leq C \|h^{\frac12} |\beta|^{\frac12} \varpi(u_h)^{\frac12} \nabla \pi_h u\|_\phi \\
& \qquad \cdot \left( \|h^{\frac12} |\beta|^{\frac12} \varpi(u_h)^{\frac12} \nabla e_h\|_\phi + K^{-1} \|e_h\|_\phi + \theta \|h^{\frac12} \mathcal{L} e_h\|_\phi \right) \\
&\leq \frac{C}{\epsilon} \|h^{\frac12} |\beta|^{\frac12} \varpi(u_h)^{\frac12} \nabla \pi_h u\|_\phi^2 \\
& \quad + \epsilon \left( \|h^{\frac12} |\beta|^{\frac12} \varpi(u_h)^{\frac12} \nabla e_h\|_\phi^2 + K^{-2} \|e_h\|^2_\phi + \|h^{\frac12} \mathcal{L} e_h\|_\phi^2 \right).
\end{align*}
Integrating in time, collecting the bounds for $I$--$III$, and noting that all $\epsilon$-weighted terms involving $e_h$, except those with $K^{-2} \|e_h\|_\phi^2$, are bounded by $\|e_h(\cdot,T)\|_\phi^2$ or parts of $\|e_h\|_{u_h,S,\phi}^2$, yields \cref{eq_weighted_continuity} after rescaling $\epsilon$.
\end{proof}
\subsection{A priori error bound}
We continue our numerical analysis with presenting a priori error bounds in a localised variant. As common in finite element a priori error bounds, the convergence behaviour, particularly the order of convergence in $h$, depends on the regularity of the solution.

Our assumption on the regularity of $u$, the solution to the continuous problem, comes in two parts: One is global in domain, i.e. it concerns both rough and smooth parts of the solution. The other is restricted to what we call smooth part of the domain, i.e. a union of simplices of the mesh $\Omega_S(t)$, where $u(t)$ is sufficiently smooth, for all times $t\in [0,T]$:
\begin{assumption}[Regularity of $u$]
Let $u$ denote the solution to the continuous problem. We assume it satisfies the following regularity conditions:
\begin{enumerate}
 \item Globally, $u$ satisfies
 \begin{align}
  u_0 \in BV(\Omega), \quad u \in L^\infty(0,T; L^\infty(\Omega)), \quad \mathrm{and} \quad \int_0^T \|\beta \cdot \nabla u\|_{\mathcal{M}(\Omega)}^2 \, \mathrm{d}t < \infty,
 \end{align}
 where $\beta \cdot \nabla u(t) \in \mathcal{M}(\Omega)$ for a.e. $t$.
 \item Locally, for each time $t \in [0,T]$, $u(t)$ satisfies on a union of simplices of the mesh $\Omega_S(t) \subset \Omega$
 \begin{align}
  u(t) \in H^{k+1} (\Omega_S(t)) \cap W^{1,\infty} (\Omega_S(t)), \quad \textnormal{and} \quad \partial_t u (t) \in H^{k}(\Omega_S(t)).
 \end{align}
\end{enumerate}
\end{assumption}
Here, $\mathcal{M}(\Omega)$ denotes the space of bounded Radon measures on $\Omega$, equipped with the total variation norm $\|\mu\|_{\mathcal{M}(\Omega)} := |\mu|(\Omega)$. For $\beta \cdot \nabla u \in L^1(\Omega)$, it holds $\|\beta \cdot \nabla u\|_{\mathcal{M}(\Omega)} = \|\beta \cdot \nabla u\|_{L^1(\Omega)}$, but the assumption also admits discontinuities of $u$ transversal to $\beta$, see \Cref{rem_measure_streamline}. Note that $\pi_h$ is defined through element-wise moments, cf.\ \cite[Section 22.3]{ern2021finite}, which require that $\beta \cdot \nabla u(t)$ has no mass on the element boundaries, i.e.\ $|\beta \cdot \nabla u(t)|(\bigcup_{T \in \mathcal{T}_h} \partial T) = 0$. For constant $\beta$, this is no restriction: as $u(x,t) = u_0(x - \beta t)$ up to contributions from $f$, the measure $\beta \cdot \nabla u(t)$ has mass on a facet $F$ only if $\beta \cdot n_F \neq 0$ and the jump set of $u_0$ has positive surface measure on $F - \beta t$, which, by $\sigma$-finiteness of the jump set, can only happen for countably many $t$. Hence the moments are well-defined for a.e. $t$, which suffices for the time-integrated estimates.
For ease of notation, we will sometimes write $\Omega_S (t)$ as $\Omega_S$ if the evaluation time is clear by context.
Furthermore, let us introduce notation $\Omega_R(t)$ for the rough part of the solution, the complement of $\Omega_S$:
\begin{equation}
 \Omega_R (t) := \Omega \ \backslash \ \Omega_S(t) \quad t \in [0,T].
\end{equation}
In terms of the mathematical language of these domains, we can now describe in detail the weighting function (c.f. the paragraph around \Cref{eq:about_varphi}) which shall be used in the a priori bound / its relation to the domains $\Omega_S, \Omega_R$ of $u$: First, let us assume the bulk interior part of the smooth domain actually has the weight $\phi$ fully active, $\phi \geq 1$. We define accordingly
\begin{equation}
 \Omega_{S,0}(t) := \{ x \in \Omega_S(t) \, | \, \phi(x,t) \geq 1 \}.
\end{equation}
In radial coordinates as used in the introduction of $\phi$, this contains the (periodically translated) ball $r \leq r_0$.

We assume furthermore that the rough part of the solution shows an asymptotically deactivated weight:
\begin{assumption}
 We assume that there exists a constant $0< C<1$ so that
 \begin{equation}
  \phi(x,t) \leq C h^{k+d/2} \quad \forall x \in \Omega_R(t), \quad t \in [0,T]. \label{eq:phi_small_in_Omega_R}
 \end{equation}
\end{assumption}
In line with the suggested scaling in radial coordinates in the construction of $\phi$, this is possible if $\mathrm{dist}(\Omega_{S,0},\Omega_R) \ge K (k + d/2) h^{1/2} |\log(h)|$.

Furthermore, we define
\begin{equation}
 \Omega_S^- := \bigcup \{ T \in \mathcal{T}_h \, | \, T \subseteq \Omega_S \textnormal{ and } \overline{T} \cap \overline{\Omega_R} = \varnothing \}.
\end{equation}
By this construction, we peel off one boundary layer of $\Omega_S$. Accordingly, we define $\Omega_R^+ = \Omega \ \backslash \ \Omega_S^-$. These definitions are illustrated in \Cref{Omega_R_S_sketch}.

\begin{figure}
 \begin{center}
   \begin{tikzpicture}[ xscale=6, yscale=6]
   \begin{scope}[Omega_S/.style={fill=Set1-B!15!white, draw=Set1-B},
     Omega_R/.style={fill=Set1-A!15!white, draw=Set1-A},
      Omega_R_plu/.style={draw=none}, Omega_S_min/.style={draw=none}
]
     \input{seq_tex/tikz_mesh_065}
     \node[Set1-A]  at (0.7,-0.1) {$\Omega_R$};
     \node[Set1-B] at (0.3,-0.1) {$\Omega_S$};
     \draw[dashed, thick] (1./3+0.005*65,0) --  (1./3+0.005*65,1);
     \begin{scope}
        \clip (0,0) rectangle (1,1);
        \draw[thick, dotted, fill=white, fill opacity=0.6] (0.1,0.5) circle (0.35);
        \node[fill=white, fill opacity=.6]  at (0.2,0.5) {$\phi > 0.9$};
        \draw[thick, dotted] (0.1,0.5) circle (0.4);
     \end{scope}
   \end{scope}
   \begin{scope}[Omega_S/.style={draw=none},
     Omega_R/.style={draw=none},
      Omega_R_plu/.style={fill=Set1-E!15!white, draw=Set1-E}, Omega_S_min/.style={fill=Set1-B!15!white, draw=Set1-B},
      xshift = 30,
]
     \input{seq_tex/tikz_mesh_065}
     \node[Set1-E]  at (0.7,-0.1) {$\Omega_R^+$};
     \node[Set1-B] at (0.3,-0.1) {$\Omega_S^-$};
     \draw[dashed, thick] (1./3+0.005*65,0) --  (1./3+0.005*65,1);
     \begin{scope}
        \clip (0,0) rectangle (1,1);
        \draw[thick, dotted] (0.1,0.5) circle (0.35);
        \draw[thick, dotted, fill=white, fill opacity=0.6] (0.1,0.5) circle (0.4);
        \node[fill=white, fill opacity=.6]  at (0.2,0.5) {$\phi > 0.1$};
     \end{scope}
   \end{scope}
   \draw[thick, dashed] (0.4,1.1) to node[right, pos=1] {shock} (0.5,1.1);
   \draw[thick, dotted] (1.2,1.1) to node[right, pos=1] {iso-lines $\phi=0.1$, $0.9$} (1.3,1.1);
  \end{tikzpicture}
 \end{center}
 \caption{Sketch of the rough and smooth domain parts for the transported shock example at $t=0.325$, $\Omega_S$ and $\Omega_R$ (left). Additionally, the construction shifted by one element layer is shown, $\Omega_S^-$ and $\Omega_R^+$ (right). To illustrate a typical intended choice of a polar coordinates weight function $\phi$ satisfying the decay assumptions, we show the iso-lines / regions of $\phi > 0.9$ and $\phi > 0.1$ within the smooth domain.}
 \label{Omega_R_S_sketch}
\end{figure}

Earlier, we denoted by $\omega_T$ the patch of neighbouring elements for any $T \in \mathcal{T}_h$. Relatedly, for a set of elements $D$, denote by $\omega(D)$ the union of patches, i.e.
\begin{equation}
 \omega(D) := \bigcup_{T \subseteq D} \omega_T.
\end{equation}
Then, it holds by construction
\begin{equation}
  \omega(\Omega_S^-) \subseteq \Omega_S. \label{moving_from_peeled_patch_to_Omega_S}
\end{equation}
Moreover, each element of $\omega(\Omega_R^+)$ lies within two element layers of $\Omega_R$. Hence, applying \cref{eq_weight_freezing} twice, the bound \cref{eq:phi_small_in_Omega_R} also holds on $\omega(\Omega_R^+)$, with a modified constant; this is used in the following without further mention.

We will make use of the following standard discrete inequality.
\begin{lemma}[Local inverse inequality for CIP stabilisation]
 For all $w_h \in V_h$, it holds
 \begin{equation}
  |\phi w_h |_s \lesssim h^{-1/2} |\beta|^{1/2}_\infty \| \phi w_h \|_\Omega. \label{eq_weighted_s_bound}
 \end{equation}
\end{lemma}
\begin{proof}
 By definition, we get using standard trace and discrete inequalities
 \begin{equation}
  | \phi w_h|_s \lesssim h |\beta|_\infty^{1/2} \| \phi \nabla w_h \|_{\mathcal{F}_h} \lesssim  h^{-1/2} |\beta|_\infty^{1/2} \| \phi w_h \|_\Omega.
 \end{equation}
\end{proof}

\begin{lemma}[Local residual conversion] \label{lemma_res_conversion}
Let $u_h$ be the solution of \cref{eq_discrete_problem}, $e_h = u_h - \pi_h u$, and assume $h^{1/2}/K$ to be sufficiently small. Then
\begin{multline}
 \sum_{T \subseteq \Omega_S^-} h^{1+d}\,\bar\phi_T^2\, R_T(u_h)^2
 \lesssim \|e_h\|_{R,\phi}^2 + |\pi_h u|_{s,\phi}^2 \\
   + \|h^{\frac12}\mathcal{L}(u - \pi_h u)\|_\phi^2
   + \rho_2 h\|\phi(f-f_h)\|_{\Omega}^2,
 \label{eq_res_conversion}
\end{multline}
with a constant depending only on $\rho_1, \rho_2$ and $|\beta|^{-1}$.
\end{lemma}
\begin{proof}
Since $\beta$ is constant, \cref{assumption_beta_const}, and $f_h \in V_h$, the facet jumps and the bulk residual entering $R_T(u_h)$, \cref{eq_def_R_T}, are polynomials on each facet and element, respectively. Hence the discrete inverse inequalities $\|q_h\|_{\infty,F}^2 \lesssim h^{1-d}\|q_h\|_F^2$ and $\|q_h\|_{\infty,T}^2 \lesssim h^{-d}\|q_h\|_T^2$, see \cite[Chapter 12]{ern2021finite}, together with \cref{eq_weight_freezing}, yield
\begin{equation*}
 h^{1+d}\bar\phi_T^2 R_T(u_h)^2 \lesssim
 \rho_1 h^2 \|\phi \jump{\nabla u_h \cdot n_F}\|_{\partial T \backslash \partial \Omega}^2
 + \rho_2 h \|\phi(\mathcal{L}u_h - f_h)\|_T^2 .
\end{equation*}
In the first term we split $u_h = e_h + \pi_h u$ and note that $\phi \jump{\nabla v_h \cdot n_F} = \jump{\nabla (\phi v_h) n_F}$ for continuous $v_h$, as $\phi$ is smooth. In the second term, using $\mathcal{L}u = f$,
\begin{equation*}
 \mathcal{L}u_h - f_h = \mathcal{L}e_h - \mathcal{L}(u - \pi_h u) + (f - f_h).
\end{equation*}
Summing over $T \subseteq \Omega_S^-$, where each facet appears at most twice, and recalling the definitions of $|\cdot|_s$ and $\|\cdot\|_R$, the claim follows.
\end{proof}
Note that the second and third summand on the right hand side of \cref{eq_res_conversion} are not small on their own, as they include the rough part of the domain. They will, however, be bounded in the proof of the main theorem below, where the necessary splitting in $\Omega_S^-$ and $\Omega_R^+$ is carried out anyway.

The main result about a priori error norms then reads as follows, for the two cases of $\alpha \in [1,2)$, and $\alpha \geq 2$:
\begin{theorem}[Weighted error estimate] \label{thm_weighted_error}
Let $K > 1$, assume that $h^{1/2}/K$ is sufficiently small, and let $\theta$ be chosen as in \Cref{prop:weight_stab}. Let $u$ satisfy the regularity assumption above, let $\phi$ satisfy \cref{eq:phi_small_in_Omega_R}, and choose the discrete initial data as $u_h(\cdot, 0) = \pi_h u_0$. Set $e_h = u_h - \pi_h u$ and, for $t \in [0,T]$,
\begin{equation*}
 \mathcal{R}(t) := \|D^{k+1} u\|_{\Omega_S,\phi}^2 + \|D^{k} \partial_t u\|_{\Omega_S,\phi}^2 + \|\beta \cdot \nabla u\|_{\mathcal{M}(\Omega)}^2 + \|u\|_{\Omega}^2 + \|f\|_{H^k(\Omega)}^2,
\end{equation*}
with all quantities evaluated at time $t$, and assume that $\mathcal{R} \in L^1(0,T)$ and $\mathcal{R}(T) < \infty$, as well as, for $\alpha \in [1,2)$, $\|\nabla u\|_{\infty,\Omega_S}^{\frac{4}{2-\alpha}} \in L^1(0,T)$. If $\alpha \in [1,2)$, then
\begin{align*}
\|e_h(\cdot,T)\|_\phi^2 + C_\theta \int_0^T \|e_h\|_{u_h,S,\phi}^2 \,\mbox{d}t
& \leq C h^{2k+1} \Big( \mathcal{R}(T) + \int_0^T \mathcal{R}(t) \,\mbox{d}t \Big)\\
& \quad + C (1 - \alpha/2) h^{\frac{2+\alpha}{2-\alpha}} \int_0^T \|\nabla u\|_{\infty,\Omega_S}^{\frac{4}{2-\alpha}} \,\mbox{d}t.
\end{align*}
If $\alpha \ge 2$, and $h \|\nabla u(\cdot,t)\|_{\infty,\Omega_S}$ is small enough for all $t$, then
\begin{equation*}
\|e_h(\cdot,T)\|_\phi^2 + C_\theta \int_0^T \|e_h\|_{u_h,S,\phi}^2 \,\mbox{d}t \leq C h^{2k+1} \Big( \mathcal{R}(T) + \int_0^T \mathcal{R}(t) \,\mbox{d}t \Big).
\end{equation*}
In both cases, the constant $C$ is independent of $h$, and depends on $\sigma_0$, $\sigma_1$, $\theta$, $U$, $|\beta|_\infty$, as well as, through Gronwall's lemma, on $T$ as $e^{CT/K^2}$.
\end{theorem}
\begin{proof}
Applying the stability of \Cref{prop:weight_stab} with $y_h = u_h$ yields, since $e_h(\cdot, 0) = 0$ by the choice of discrete initial data,
\begin{alignat*}{1}
\|e_h(\cdot,T)\|_\phi^2  & + C_\theta \int_0^T  \|e_h\|_{u_h,S,\phi}^2
~\mbox{d}t    \leq C/K^2  \int_0^T \|e_h\|_\phi^2 ~\mbox{d}t
 \nonumber\\
&+ 2 \int_0^T( (\mathcal{L} e_h, w_h)_\Omega + \sigma_0 s_0(u_h; e_h,w_h) + \sigma_1 s_1(u_h;e_h,w_h) )~\mbox{d}t 
\end{alignat*}
where $w_h = \pi_h \phi^2 (e_h + \theta h\,i_{av}\mathcal{L} e_h)$.
By \Cref{lemma_weighted_continuity}, with $\epsilon$ small enough such that $2\epsilon \leq \frac12 \min(1, C_\theta)$, the contribution $2\epsilon \|e_h\|_{u_h,\phi}^2$ is absorbed into the left hand side, and we obtain
\begin{alignat}{1}
\|e_h(\cdot,T)\|_\phi^2  & + \frac{C_\theta}{2} \int_0^T  \|e_h\|_{u_h,S,\phi}^2
~\mbox{d}t    \leq \frac{C}{K^2}  \int_0^T \|e_h\|_\phi^2 ~\mbox{d}t
 \nonumber\\
&+ C \int_0^T(\|h^{\frac12} \mathcal{L} (u - \pi_h u) \|_\phi^2+ |\pi_h u|_{s,\phi}^2 + \|h^{\frac12} \varpi(u_h)^{\frac12} \nabla \pi_h u \|_\phi^2 )~\mbox{d}t \label{eq_large_step_in_error_bnd} \\
& + C \int_0^T { \|h^{-\frac12} (u - \pi_h u)\|_\phi^2 } ~\mbox{d}t + C \|(u - \pi_h u)(\cdot, T)\|_\phi^2. \nonumber
\end{alignat}
To conclude we need to bound the interpolation terms on the right hand side, respecting the different regularities in the subdomains $\Omega_S$ and $\Omega_R$.
For the first term we observe that
\[
\|h^{\frac12} \mathcal{L} (u - \pi_h u) \|_\phi^2 \leq \|h^{\frac12} \mathcal{L} (u - \pi_h u) \|_{\Omega_S^-,\phi}^2 + \|h^{\frac12} \mathcal{L} (u - \pi_h u) \|_{\Omega_R^+,\phi}^2.
\]
The integral over the smooth part is handled using the local interpolation estimate \cref{EGint_interpol}, noting that $\omega(\Omega_S^-) \subseteq \Omega_S$ by \cref{moving_from_peeled_patch_to_Omega_S}, that $\phi$ varies at most by a constant factor on each patch by \cref{eq_weight_freezing}, and that $\partial_t \pi_h u = \pi_h \partial_t u$,
\[
\|h^{\frac12} \mathcal{L} (u - \pi_h u) \|_{\Omega_S^-,\phi}^2 \leq C h^{2k+1} (\|D^{k+1} u\|_{\Omega_S,\phi}^2 + \|D^{k} \partial_t u\|_{\Omega_S,\phi}^2).
\]
Using now that $\mathcal{L}u = f$ and $\|\partial_t \pi_h u\|_\phi = \|\pi_h \partial_t u\|_\phi = \| \pi_h (f -  \beta \cdot \nabla u) \|_\phi$ we see that
\[
\|h^{\frac12} \mathcal{L} (u - \pi_h u) \|_{\Omega_R^+,\phi}^2 \leq
C (\|h^{\frac12} f \|_{\Omega_R^+,\phi}^2 + \|h^{\frac12} \pi_h \beta \cdot \nabla u\|_{\Omega_R^+,\phi}^2 + \|h^{\frac12}  \beta \cdot \nabla \pi_h u\|_{\Omega_R^+,\phi}^2).
\]
For the first summand, \cref{eq:phi_small_in_Omega_R} directly yields $\|h^{\frac12} f \|_{\Omega_R^+,\phi}^2 \leq C h^{2k+1} \|f\|_{\Omega}^2$. For the second summand, applying the inverse inequality $\|u_h\|_{L^2(\Omega)} \leq C h^{-d/2} \|u_h\|_{L^1(\Omega)}$, the bound on $\phi$ in $\Omega_R$, and the local $L^1$-stability \cref{EGint_Lpstable} of $\pi_h$, we bound the right hand side. Here, $\pi_h$ is applied to the measure $\beta \cdot \nabla u$ through its element-wise moments against polynomials, which are well-defined as $\beta \cdot \nabla u$ has no mass on the element boundaries; since these moments satisfy $|\langle \mu, q \rangle_T| \leq |\mu|(T) \|q\|_{\infty,T}$, the local stability \cref{EGint_Lpstable} for $p=1$ carries over with $\|\cdot\|_{L^1(\omega_T)}$ replaced by $|\cdot|(\omega_T)$. Therefore,
\[
\|h^{\frac12} \pi_h \beta \cdot \nabla u\|_{\Omega_R^+,\phi}^2 \leq C h^{2k+1} \| \pi_h \beta \cdot \nabla u\|_{L^1(\Omega_R^+)}^2 \leq C h^{2k+1} |\beta \cdot \nabla u|(\omega(\Omega_R^+))^2,
\]
and using a standard inverse inequality and the local $L^2$-stability of $\pi_h$, \cref{EGint_Lpstable},
\[
\|h^{\frac12}  \beta \cdot \nabla \pi_h u\|_{\Omega_R^+,\phi}^2 \leq C h^{2k+1} \|u\|_{\omega(\Omega_R^+)}^2.
\]
As it relates to the $h^{-\frac12}$-weighted interpolation error in \Cref{eq_large_step_in_error_bnd}, we argue similarly:
\begin{equation}
 { \|h^{-\frac12} (u - \pi_h u)\|_\phi^2 } \lesssim \|h^{-\frac12} (u - \pi_h u)\|_{\Omega_S^-, \phi}^2 + \|h^{-\frac12} (u - \pi_h u)\|_{\Omega_R^+, \phi}^2.
\end{equation}
Again, by \cref{EGint_interpol} and \cref{moving_from_peeled_patch_to_Omega_S},
\begin{equation}
 \|h^{-\frac12} (u - \pi_h u)\|_{\Omega_S^-, \phi}^2 \lesssim h^{2k+1} \|D^{k+1} u\|_{\Omega_S,\phi}^2.
\end{equation}
On the other side, by \cref{eq:phi_small_in_Omega_R},
\begin{equation}
 \|h^{-\frac12} (u - \pi_h u)\|_{\Omega_R^+, \phi}^2 \lesssim h^{-1} |\phi|_{\infty, \Omega_R^+}^2 \| u \|_{\omega(\Omega_R^+)}^2 \lesssim h^{2k+1} \| u \|_{\omega(\Omega_R^+)}^2.
\end{equation}
The same arguments at time $t = T$, with an additional factor $h$, show that the endpoint term in \Cref{eq_large_step_in_error_bnd} satisfies $\|(u - \pi_h u)(\cdot, T)\|_\phi^2 \leq C h^{2k+1} \mathcal{R}(T)$.
For the gradient jump term we obtain similarly using approximation in $\Omega_S^-$ and \cref{eq_weighted_s_bound} in $\Omega_R^+$ and the bound on $\phi$, \cref{eq:phi_small_in_Omega_R},
\begin{align}\label{eq:1st_rough}
|\pi_h u|_{s,\phi}^2 &\leq C h^{2k+1} \|D^{k+1} u\|_{\Omega_S,\phi}^2 + C \|h^{-\frac12} \pi_h u\|_{\omega(\Omega_R^+),\phi}^2 \\
&\leq C h^{2k+1} \|D^{k+1} u\|_{\Omega_S,\phi}^2 + C h^{2k + 1} \|u\|_{\omega(\Omega_R^+)}^2.
\end{align}
Finally, considering the term stemming from the nonlinear stabilization we have to proceed with some care. As before we divide it in $\Omega_S^-$ and $\Omega_R^+$,
\[
\|h^{\frac12} \varpi(u_h)^{\frac12} \nabla \pi_h u \|_\phi^2 \leq \|h^{\frac12} \varpi(u_h)^{\frac12} \nabla \pi_h u \|_{\Omega_S^-,\phi}^2 + \|h^{\frac12} \varpi(u_h)^{\frac12} \nabla \pi_h u \|_{\Omega_R^+,\phi}^2.
\]
Starting with the second term we see that using an inverse inequality and arguments similar to those of \eqref{eq:1st_rough},
\[
\|h^{\frac12} \varpi(u_h)^{\frac12} \nabla \pi_h u \|_{\Omega_R^+,\phi}^2 \leq C h^{2k + 1} \|u\|_{\omega(\Omega_R^+)}^2.
\]
The first term must be handled differently depending on the parameter $\alpha$ in $\varpi(u_h)$. In both cases, we use $\phi \leq \bar\phi_T$ on $T$ by definition to obtain
\[
\|h^{\frac12} \varpi(u_h)^{\frac12} \nabla \pi_h u \|_{\Omega_S^-,\phi}^2 \leq \sum_{T \subseteq \Omega_S^-} h\, \varpi(u_h)|_T\, \bar\phi_T^2 \|\nabla \pi_h u\|_T^2 .
\]
Moreover, with $x_T := h R_T(u_h)/U$ we have $\varpi(u_h)|_T = \min(1, x_T)^\alpha$. Hence, as $\min(1,x_T) \leq x_T$,
\begin{equation}
 \varpi(u_h)|_T^{2/\alpha} \leq \frac{h^2 R_T(u_h)^2}{U^2}, \qquad \textnormal{and, if } \alpha \geq 2, \quad \varpi(u_h)|_T \leq \frac{h^2 R_T(u_h)^2}{U^2}. \label{eq_varpi_by_R}
\end{equation}
Furthermore, we recall that by the local $W^{1,\infty}$-stability \cref{EGint_w1stable} together with \cref{moving_from_peeled_patch_to_Omega_S}, $\|\nabla \pi_h u\|_{\infty,\Omega_S^-} \lesssim \|\nabla u\|_{\infty,\Omega_S}$.

First assume that $\alpha \in [1,2)$. We follow \cite[Lemma 5]{burman2023someobservations}, incorporating the weight: Young's inequality $ab \leq p^{-1} a^p + q^{-1} b^q$ with $p = 2/\alpha$, $q = 2/(2-\alpha)$, applied elementwise with
\[
a = h^{\alpha(d-1)/2} \varepsilon^{\alpha/2} \bar\phi_T^\alpha\, \varpi(u_h)|_T, \qquad b = \varepsilon^{-\alpha/2} h^{1-\alpha(d-1)/2} \bar\phi_T^{2-\alpha} \|\nabla \pi_h u\|_T^2,
\]
together with \cref{eq_varpi_by_R}, $\bar\phi_T \leq 2$, and the scaling argument of \cite[Lemma 5]{burman2023someobservations} for the second summand, yields
\begin{multline*}
\sum_{T \subseteq \Omega_S^-} h\, \varpi(u_h)|_T\, \bar\phi_T^2 \|\nabla \pi_h u\|_T^2 \leq \frac{\alpha}{2} \frac{\varepsilon}{U^2} \sum_{T \subseteq \Omega_S^-} h^{1+d}\, \bar\phi_T^2 R_T(u_h)^2 \\
+ C (1 - \alpha/2)\, \varepsilon^{-q\alpha/2} h^{\frac{2+\alpha}{2-\alpha}} \|\nabla u\|_{\infty,\Omega_S}^{\frac{4}{2-\alpha}}.
\end{multline*}
Applying \Cref{lemma_res_conversion} to the first summand, we arrive at
\begin{multline*}
\|h^{\frac12} \varpi(u_h)^{\frac12} \nabla \pi_h u \|_{\Omega_S^-,\phi}^2 \leq C \varepsilon \big( \|e_h\|_{R,\phi}^2 + |\pi_h u|_{s,\phi}^2 \\
+ \|h^{\frac12}\mathcal{L}(u - \pi_h u)\|_\phi^2 + \rho_2 h\|\phi(f-f_h)\|_{\Omega}^2 \big) + C (1 - \alpha/2) h^{\frac{2+\alpha}{2-\alpha}} \|\nabla u\|_{\infty,\Omega_S}^{\frac{4}{2-\alpha}},
\end{multline*}
with $C$ depending on $U$ and $\varepsilon$. For $\varepsilon$ small enough, the first summand is absorbed into the left hand side. The second and third summand have been bounded above by $C h^{2k+1} \mathcal{R}(t)$, and by the interpolation property of $f_h$, $h\|\phi(f-f_h)\|_{\Omega}^2 \leq C h^{2k+1} \|f\|_{H^k(\Omega)}^2$. The estimate for $\alpha \in [1,2)$ now follows after collecting terms and applying Gronwall's lemma.

If on the other hand $\alpha \geq 2$, we use \cref{eq_varpi_by_R} and $\|\nabla \pi_h u\|_T^2 \lesssim h^d \|\nabla \pi_h u\|_{\infty,T}^2$ to obtain, again by \Cref{lemma_res_conversion},
\begin{align*}
\|h^{\frac12} \varpi(u_h)^{\frac12} \nabla \pi_h u \|_{\Omega_S^-,\phi}^2 &\lesssim \frac{h^2}{U^2} \|\nabla u\|_{\infty,\Omega_S}^2 \sum_{T \subseteq \Omega_S^-} h^{1+d}\, \bar\phi_T^2 R_T(u_h)^2 \\
&\leq C h^2 \|\nabla u\|_{\infty,\Omega_S}^2 \big( \|e_h\|_{R,\phi}^2 + |\pi_h u|_{s,\phi}^2 \\
&\qquad + \|h^{\frac12}\mathcal{L}(u - \pi_h u)\|_\phi^2 + \rho_2 h\|\phi(f-f_h)\|_{\Omega}^2 \big).
\end{align*}
Assuming that $h$ is small enough so that $C h^2 \|\nabla u\|_{\infty,\Omega_S}^2 \leq \frac12 C_\theta$, the first summand is absorbed into the left hand side, and the remaining summands are bounded as in the case $\alpha \in [1,2)$. The estimate for $\alpha \geq 2$ follows again after collecting terms and applying Gronwall's lemma.
\end{proof}

Since $\phi \geq 1$ on $\Omega_{S,0}$, we obtain the following unweighted error bound in the interior of the smooth part of the domain, which is the localisation principle announced in the introduction.
\begin{corollary}[Local $L^2$ error] \label{cor_local_L2}
Under the assumptions of \Cref{thm_weighted_error}, and for $h$ sufficiently small such that all elements intersecting $\Omega_{S,0}(T)$, together with their patches, are contained in $\Omega_S(T)$, it holds
\begin{equation*}
 \|(u - u_h)(\cdot, T)\|_{\Omega_{S,0}(T)} \leq C h^{k+\frac12} + C (1-\alpha/2)^{\frac12} h^{\frac{2+\alpha}{2(2-\alpha)}},
\end{equation*}
where the second summand is only present for $\alpha \in [1,2)$, and $C$ depends on $u$ and $f$ through $\mathcal{R}$ and $\|\nabla u\|_{\infty,\Omega_S}$ as in \Cref{thm_weighted_error}.
\end{corollary}
\begin{proof}
Since $\phi \geq 1$ on $\Omega_{S,0}$, the triangle inequality gives $\|(u - u_h)(\cdot,T)\|_{\Omega_{S,0}} \leq \|(u - \pi_h u)(\cdot,T)\|_{\Omega_{S,0}} + \|e_h(\cdot,T)\|_\phi$. The first summand is of order $h^{k+1}$ by \cref{EGint_interpol}, and the second one is bounded by \Cref{thm_weighted_error}.
\end{proof}
Note that the condition on $h$ in \Cref{cor_local_L2} is not restrictive, as $\mathrm{dist}(\Omega_{S,0}, \Omega_R) \geq K (k + d/2) h^{1/2} |\log h|$ is much larger than the width of two element layers for small $h$.
\begin{remark}[Choice of $\alpha$]
Since $\frac{2+\alpha}{2-\alpha} \geq 2k+1$ if and only if $\alpha \geq \frac{2k}{k+1}$, \Cref{cor_local_L2} yields the rate $\mathcal{O}(h^{k+1/2})$ for all $\alpha \geq \frac{2k}{k+1}$, in particular for all $\alpha \geq 2$. For smaller values of $\alpha$, the second summand limits the rate to $\mathcal{O}(h^{\frac{2+\alpha}{2(2-\alpha)}})$; e.g., for $\alpha = 1$, the rate is $\mathcal{O}(h^{3/2})$ independently of $k$.
\end{remark}
\begin{remark}[Streamline derivative as a measure] \label{rem_measure_streamline}
The streamline derivative enters the error bound only through the term $\|h^{\frac12}\pi_h(\beta\cdot\nabla u)\|_{\Omega_R^+,\phi}$ in the rough part of the domain. Its control in $\mathcal{M}(\Omega)$ instead of $L^1(\Omega)$ is essential for the intended applications: for a discontinuity of $u$ across a surface $\Gamma$ transversal to $\beta$, such as the shock in \Cref{figure_motivating_example}, $\beta\cdot\nabla u = \jump{u}\,(\beta\cdot n_\Gamma)\,\mathcal{H}^{d-1}|_\Gamma$ is not a function, but a bounded measure with $\|\beta\cdot\nabla u\|_{\mathcal{M}(\Omega)} \leq |\beta| \, \|\jump{u}\|_{L^1(\Gamma)}$. With the $L^2$-projection instead of $\pi_h$, as in \cite{burman2022weighted}, one could even control the streamline derivative in $(H^1(\Omega))'$, since then $\|P_h z\|_\Omega \leq C h^{-1}\|z\|_{(H^1(\Omega))'}$ and $\|\beta\cdot\nabla u\|_{(H^1(\Omega))'} \leq |\beta|\,\|u\|_\Omega$. However, $\pi_h$ is defined through element-wise moments and is not stable in $(H^1(\Omega))'$, which is the price for its local $W^{1,\infty}$-stability used in the proof of \Cref{thm_weighted_error}.
\end{remark}
\subsection*{Acknowledgements}
The authors acknowledge support by EPSRC under the grant  EP/X042650/1. FH was supported by the Swedish Research Council under grant no. 2021-06594 while in residence at Institut Mittag-Leffler in Djursholm, Sweden during the fall semester of 2025.
\subsection*{Statement on the use of AI}
Anthropic Claude and OpenAI ChatGPT Astra were used in the final completion of this article, following up on the manual first preprint version. Any incorporated material has been checked and amended by the authors, who assume responsibility for the arguments put forward.
\bibliographystyle{siamplain}
\bibliography{paper}

\begin{thebibliography}{10}

\bibitem{BERTOLUZZA19991097}
{\sc S.~Bertoluzza}, {\em The discrete commutator property of approximation
  spaces}, Comptes Rendus de l'Académie des Sciences - Series I - Mathematics,
  329 (1999), pp.~1097--1102,
  \url{https://doi.org/10.1016/S0764-4442(00)88481-5},
  \url{https://www.sciencedirect.com/science/article/pii/S0764444200884815}.

\bibitem{brennerscott}
{\sc S.~C. Brenner and L.~R. Scott}, {\em The mathematical theory of finite
  element methods}, Springer, 2008,
  \url{https://doi.org/10.1007/978-0-387-75934-0}.

\bibitem{BROOKS1982199}
{\sc A.~N. Brooks and T.~J. Hughes}, {\em Streamline upwind/petrov-galerkin
  formulations for convection dominated flows with particular emphasis on the
  incompressible navier-stokes equations}, Computer Methods in Applied
  Mechanics and Engineering, 32 (1982), pp.~199--259,
  \url{https://doi.org/10.1016/0045-7825(82)90071-8},
  \url{https://www.sciencedirect.com/science/article/pii/0045782582900718}.

\bibitem{burman2022weighted}
{\sc E.~Burman}, {\em Weighted error estimates for transient transport problems
  discretized using continuous finite elements with interior penalty
  stabilization on the gradient jumps}, Vietnam Journal of Mathematics, 50
  (2022), pp.~833--866, \url{https://doi.org/10.1007/s10013-022-00550-x}.

\bibitem{burman2023someobservations}
{\sc E.~Burman}, {\em Some observations on the interaction between linear and
  nonlinear stabilization for continuous finite element methods applied to
  hyperbolic conservation laws}, SIAM Journal on Scientific Computing, 45
  (2023), pp.~A96--A122, \url{https://doi.org/10.1137/21M1464154}.

\bibitem{M2AN_2007__41_1_55_0}
{\sc E.~Burman and A.~Ern}, {\em A continuous finite element method with face
  penalty to approximate {Friedrichs'} systems}, ESAIM: Mod\'elisation
  math\'ematique et analyse num\'erique, 41 (2007), pp.~55--76,
  \url{https://doi.org/10.1051/m2an:2007007},
  \url{https://www.numdam.org/articles/10.1051/m2an:2007007/}.

\bibitem{burman2007continuous}
{\sc E.~Burman and A.~Ern}, {\em Continuous interior penalty hp-finite element
  methods for advection and advection-diffusion equations}, Mathematics of
  Computation, 76 (2007), pp.~1119--1140.

\bibitem{10.1093/imanum/drn001}
{\sc E.~Burman, J.~Guzmán, and D.~Leykekhman}, {\em Weighted error estimates
  of the continuous interior penalty method for singularly perturbed problems},
  IMA Journal of Numerical Analysis, 29 (2008), pp.~284--314,
  \url{https://doi.org/10.1093/imanum/drn001},
  \url{https://doi.org/10.1093/imanum/drn001},
  \url{https://arxiv.org/abs/https://academic.oup.com/imajna/article-pdf/29/2/284/1930247/drn001.pdf}.

\bibitem{BURMAN20041437}
{\sc E.~Burman and P.~Hansbo}, {\em Edge stabilization for galerkin
  approximations of convection–diffusion–reaction problems}, Computer
  Methods in Applied Mechanics and Engineering, 193 (2004), pp.~1437--1453,
  \url{https://doi.org/10.1016/j.cma.2003.12.032},
  \url{https://www.sciencedirect.com/science/article/pii/S004578250400043X}.
\newblock Recent Advances in Stabilized and Multiscale Finite Element Methods.

\bibitem{doi:10.1137/060668936}
{\sc B.~Cockburn and J.~Guzm{\'a}n}, {\em Error estimates for the runge–kutta
  discontinuous galerkin method for the transport equation with discontinuous
  initial data}, SIAM Journal on Numerical Analysis, 46 (2008), pp.~1364--1398,
  \url{https://doi.org/10.1137/060668936},
  \url{https://doi.org/10.1137/060668936},
  \url{https://arxiv.org/abs/https://doi.org/10.1137/060668936}.

\bibitem{di2011mathematical}
{\sc D.~A. Di~Pietro and A.~Ern}, {\em Mathematical aspects of discontinuous
  Galerkin methods}, vol.~69, Springer Science \& Business Media, 2011,
  \url{https://doi.org/10.1007/978-3-642-22980-0}.

\bibitem{10.1007/BFb0120591}
{\sc J.~Douglas and T.~Dupont}, {\em Interior penalty procedures for elliptic
  and parabolic galerkin methods}, in Computing Methods in Applied Sciences,
  R.~Glowinski and J.~L. Lions, eds., Berlin, Heidelberg, 1976, Springer Berlin
  Heidelberg, pp.~207--216.

\bibitem{doi:10.1137/120867482}
{\sc A.~Ern and J.-L. Guermond}, {\em Weighting the edge stabilization}, SIAM
  Journal on Numerical Analysis, 51 (2013), pp.~1655--1677,
  \url{https://doi.org/10.1137/120867482},
  \url{https://doi.org/10.1137/120867482},
  \url{https://arxiv.org/abs/https://doi.org/10.1137/120867482}.

\bibitem{ern2021finite}
{\sc A.~Ern and J.-L. Guermond}, {\em Finite Elements I}, Springer Cham, 2021,
  \url{https://doi.org/10.1007/978-3-030-56341-7}.

\bibitem{ern2021finite3}
{\sc A.~Ern and J.-L. Guermond}, {\em Finite Elements III}, Springer Cham,
  2021, \url{https://doi.org/10.1007/978-3-030-57348-5}.

\bibitem{hughes2012finite}
{\sc T.~J. Hughes}, {\em The finite element method: linear static and dynamic
  finite element analysis}, Courier Corporation, 2012.

\bibitem{layton1981estimates}
{\sc W.~J. Layton}, {\em Estimates away from a discontinuity for dissipative
  galerkin methods for hyperbolic equations}, Mathematics of Computation, 36
  (1981), pp.~87--92.

\bibitem{zhou1995local}
{\sc G.~Zhou}, {\em A local $l^{2}$-error analysis of the streamline diffusion
  method for nonstationary convection-diffusion systems}, ESAIM:
  Mod{\'e}lisation math{\'e}matique et analyse num{\'e}rique, 29 (1995),
  pp.~577--603.

\end{thebibliography}
\end{document}